\documentclass[11pt,a4paper,leqno]{article}
\usepackage{a4wide}
\usepackage{latexsym}
\usepackage{amsmath}
\usepackage{amssymb}
\newtheorem{defin}{Definition}
\newtheorem{lemma}{Lemma}  
\newtheorem{prop}{Proposition}
\newtheorem{theo}{Theorem}
\newtheorem{corol}{Corollary}
\newenvironment{proof}{\medskip\par\noindent{\bf Proof}}{\hfill $\Box$
\medskip\par}

\def\a{\alpha} 
\def\b{\beta} 
\def\C{\mathbb{C}} 
\def\N{\mathbb{N}}
\def\R{\mathbb{R}}  
\def\Z{\mathbb{Z}}

\begin{document}



\title{On $q-$Gevrey asymptotics for singularly perturbed $q-$difference-differential problems with an irregular singularity}
\author{Alberto Lastra, St\'ephane Malek}

\maketitle

\begin{center}
{\bf Abstract}
\end{center}

We study a $q-$analog of a singularly perturbed Cauchy problem with irregular singularity in the complex domain which generalizes a previous result by S. Malek in~\cite{malek}. First, we construct solutions defined in open $q-$spirals to the origin. By means of a $q-$Gevrey version of Malgrange-Sibuya theorem we show the existence of a formal power series in the perturbation parameter which turns out to be the $q-$Gevrey asymptotic expansion (of certain type) of the actual solutions.


\begin{section}{Introduction}
We study a family of $q$-difference-differential equations of the following form
\begin{equation}\label{eq1}\epsilon t \partial_{z}^{S}X(\epsilon,qt,z)+\partial_z^{S}X(\epsilon,t,z)=\sum_{k=0}^{S-1}b_k(\epsilon,z)(t\sigma_q)^{m_{0,k}}(\partial_z^kX)(\epsilon,t,zq^{-m_{1,k}}),
\end{equation}
where $q\in\C$ such that $|q|>1$, $m_{0,k},m_{1,k}$ are positive integers, $b_{k}(\epsilon,z)$ are polynomials in $z$ with holomorphic coefficients in $\epsilon$ on some neighborhood of 0 in $\C$ and $\sigma_{q}$ is the dilation operator given by $(\sigma_{q}X)(\epsilon,t,z)=X(\epsilon,qt,z)$. As in previous works~\cite{malek1},~\cite{malek2},~\cite{lastramaleksanz}, the map $(t,z)\mapsto (q^{m_{0,k}}t,zq^{-m_{1,k}})$ is assumed to be a volume shrinking map, meaning that the modulus of the Jacobian determinant $|q|^{m_{0,k}-m_{1,k}}$ is less than 1, for every $0\le k\le S-1$.

In~\cite{malek}, the second author studies a similar singularly perturbed Cauchy problem. In this previous work, the polynomial $b_{k}(\epsilon,z):=\sum_{s\in I_{k}}b_{ks}(\epsilon)z^s$ is such that, for all $0\le k\le S-1$, $I_{k}$ is a finite subset of $\N=\{0,1,...\}$ and $b_{ks}(\epsilon)$ are bounded holomorphic functions on some disc $D(0,r_{0})$ in $\C$ which verify that the origin is a zero of order at least $m_{0,k}$. The main point on these flatness conditions on the coefficients in $b_{k}(\epsilon,z)$ is that the method used by M. Canalis-Durand, J. Mozo-Fern\'andez and R. Sch\"{a}fke in~\cite{canalisjorge} could be adapted so that the initial singularly perturbed problem turns into an auxiliar regularly perturbed $q-$difference$-$differential equation with an irregular singularity at $t=0$, preserving holomorphic coefficients $b_{ks}$ (we refer to~\cite{malek} for the details). These constricting conditions on the flatness of $b_{k}(\epsilon,z)$ is now omitted, so that previous result is generalized. In the present work we will not only make use of the procedure considered in~\cite{canalisjorge} but also of the methodology followed in~\cite{malek0}. In that work, the second author considers a family of singularly perturbed nonlinear partial differential equations such that the coefficients appearing possess poles with respect to $\epsilon$ at the origin after the change of variable $t\mapsto t/\epsilon$. This scenary fits our problem.

In both, the present work and~\cite{malek0}, the procedure for locating actual solutions relies on the research of certain appropriate Banach spaces. The ones appearing here may be regarded as $q-$analogs of the ones in~\cite{malek0}.  

In order to fix ideas we first settle a brief summary of the procedure followed. We consider a finite family of discrete $q-$spirals $(U_{I}q^{-\N})_{I\in\mathcal{I}}$ in such a way that it provides a good covering at 0 (Definition~\ref{def527}). 

We depart from a finite family, with indices belonging to a set $\mathcal{I}$, of perturbed Cauchy problems (\ref{ee1})+(\ref{ii1}). Let $I\in\mathcal{I}$ be fixed. Firstly, by means of a non-discrete $q-$analog of Laplace transform introduced by C. Zhang in~\cite{zhang1} (for details on classical Laplace transform we refer to~\cite{balser},\cite{costin}), we are able to transform our initial problem into auxiliary equation (\ref{fe1}) (or (\ref{fe2})). 

The transformed problem fits into certain Cauchy auxiliar problem such as (\ref{fe1})+(\ref{ic1}) which is considered in Section~\ref{seccion1}. Here, its solution is found in the space of formal power series in $z$ with coefficients belonging to the space of holomorphic functions defined in the product of discrete $q-$spirals to the origin in the variable $\epsilon$ (this domain corresponds to $U_{I}q^{-\N}$ in the auxiliar transformed problem) times a continuous $q-$spiral to infinity in the variable $\tau$ ($V_{I}q^{\R_{+}}$ for the auxiliar equation). Moreover, for any fixed $\epsilon$ and regarding our auxiliar equation, one can deduce that the coefficients, as functions in the variable $\tau$, belong to the Banach space of holomorphic functions in $V_{I}q^{\R_+}$ subjected to $q-$Gevrey bounds
$$ |W_{\b}^{I}(\epsilon,\tau)|\le C_{1}\b!H^{\b}e^{M\log^{2}|\tau/\epsilon|}\left|\frac{\tau}{\epsilon}\right|^{C\b}|q|^{-A_{1}\b^2},\quad \tau\in V_{I}q^{\R_+}$$
for positive constants $C_{1},C,M,H, A_{1}>0$, where the index of the coefficient considered is $\b$ (see Theorem~\ref{teorema1}).

Also, the transformed problem fits into the auxiliar problem (\ref{fe2})+(\ref{ic2}), studied in details in Section~\ref{seccion2}. In this case, the solution is found in the space of formal power series in $z$ with coefficients belonging to the space of holomorphic functions defined in the product of a punctured disc at 0 in the variable $\epsilon$ times a punctured disc at the origin in $\tau$. For a fixed $\epsilon$, the coefficients belong to the Banach space of holomorphic functions in $D(0,\rho_0)\setminus\{0\}$ such that 
$$|W_{\b}^{I}(\epsilon,\tau)|\le C_{1}\b!H^{\b}e^{M\log^{2}|\tau/\epsilon|}|\epsilon|^{-C\b}|q|^{-A_{1}\b^2},\quad\tau\in D(0,\rho_{0})\setminus\{0\}$$
for positive constants $C_{1},C,M,H, A_{1}>0$ when $\b$ is the coefficient considered (see Theorem~\ref{teorema2}). 

From these results, we get a sequence $(W_{\b}^{I})_{\b\in\N}$ consisting of holomorphic functions in the variable $\tau$ so that $q-$Laplace transform can be applied to its elements. In addition, the function
\begin{equation}\label{e10}
X_I(\epsilon,t,z):=\sum_{\b\ge0}\mathcal{L}^{\lambda_I}_{q;1}W_{\b}^{I}(\epsilon,\epsilon t)\frac{z^{\b}}{\b!}
\end{equation}
turns out to be a holomorphic function defined in $U_{I}q^{-\N}\times \mathcal{T}\times \C$ which is a solution of the initial problem. Here, $\mathcal{T}$ is an adequate open half $q-$spiral to 0 and $\lambda_{I}$ corresponds to certain $q-$directions for the $q-$Laplace transform (see Proposition~\ref{prop11}). The way to proceed is also followed by the authors in~\cite{vizio1} and~\cite{vizio2} when studying asymptotic properties of analytic solutions of $q-$difference equations with irregular singularities.

It is worth pointing out that the choice of a continuous summation procedure unlike the discrete one in~\cite{malek} is due to the requirement of Cauchy's theorem on the way.

At this point we own a finite family $(X_{I})_{I\in\mathcal{I}}$ of solutions of (\ref{ee1})+(\ref{ii1}). The main goal is to study its asymptotic behavior at the origin in some sense. Let $\rho>0$. One observes (Theorem~\ref{teo572}) that whenever the intersection $U_{I}\cap U_{I'}$ is not empty we have
\begin{equation}\label{e11}|X_{I}(\epsilon, t, z )-X_{I'}(\epsilon, t, z)|\le C_{1}e^{-\frac{1}{A}\log^2|\epsilon|}
\end{equation}
for positive constants $C_{1},A$ and for every $(\epsilon, t,z)\in (U_{I}q^{-\N}\cap U_{I'}q^{-\N})\times \mathcal{T}\times D(0,\rho)$. Equation (\ref{e11}) implies that the difference of two solutions of (\ref{ee1})+(\ref{ii1}) admits $q-$Gevrey null expansion of type $A>0$ at 0 in $U_{I}\cap U_{I'}$ as a function with values in the Banach space $\mathbb{H}_{\mathcal{T},\rho}$ of holomorphic bounded functions defined in $\mathcal{T}\times D(0,\rho)$ endowed with the supremum norm. Flatness condition (\ref{e11}) allows us to establish the main result of the present work (Theorem~\ref{teopral}): the existence of a formal power series 
$$\hat{X}(\epsilon)=\sum_{k\ge 0}\frac{X_{k}}{k!}\epsilon^{k}\in\mathbb{H}_{\mathcal{T},\rho}[[\epsilon]],$$
formal solution of (\ref{eq1}), such that for every $I\in\mathcal{I}$, each of the actual solutions (\ref{e10}) of the problem (\ref{ee1})+(\ref{ii1}) admits $\hat{X}$ as its $q-$Gevrey expansion of a certain type in the corresponding domain of definition.

The main result heavily rests on  a Malgrange-Sibuya type theorem involving $q-$Gevrey bounds, which generalizes a result in~\cite{malek} where no precise bounds on the asymptotic appears. In this step, we make use of Whitney-type extension results in the framework of ultradifferentiable functions. Whitney-type extension theory is widely studied in literature under the framework of ultradifferentiable functions subjected to bounds of their derivatives (see for example~\cite{chch},~\cite{bonetbraunmeisetaylor}) and also it is a useful tool taken into account on the study of continuity of ultraholomorphic operators (see~\cite{javi},\cite{thilliez},\cite{lastrasanz}). It is also worth saying that, although $q-$Gevrey bounds have been achieved in the present work, the type involved might be increased when applying an extension result for ultradifferentiable functions from~\cite{bonetbraunmeisetaylor}.

The paper is organized as follows.\\
In Section 2 and Section 3, we introduce Banach spaces of formal power series and solve auxiliary Cauchy problems involving these spaces. In Section 2, this is done when the variables rely in a product of a discrete $q-$spiral to the origin times a $q-$spiral to infinity, while in Section 3 it is done when working on a product of a punctured disc at 0 times a disc at 0.

In Section 4 we first recall definitions and some properties related to $q-$Laplace transform appearing in~\cite{zhang1}, firstly developed by C. Zhang. In this section we also find actual solutions of the main Cauchy problem (\ref{ee1})+(\ref{ii1}) and settle a flatness condition on the difference of two of them so that, when regarding the difference of two solutions in the variable $\epsilon$, we are able to give some information on its asymptotic behavior at 0. Finally, in Section 6 we conclude with the existence of a formal power series in $\epsilon$ with coefficients in an adequate Banach space of functions which solves in a formal sense the problem considered. The procedure heavily rests on a $q-$Gevrey version of Malgrange-Sibuya theorem, developed in Section 5.

\end{section}

\begin{section}{A Cauchy problem in weighted Banach spaces of Taylor series}\label{seccion1}

$M,A_{1},C>0$ are fixed positive real numbers throughout the whole paper.

Let $U,V$ be bounded sets in $\C^{\star}$ and let $q\in\C^{\star}$ such that $|q|>1$.  We define 
$$ Uq^{-\N}=\{\epsilon q^{-n}\in\C: \epsilon\in U,n\in\N\}\quad,\quad Vq^{\R_+}=\{\tau q^{l}\in\C:\tau\in V,l\in\R,l\ge0\}.$$
We assume there exists $M_{1}>0$ such that $|\tau+1|>M_{1}$ for all $\tau\in Vq^{\R_+}$ and also that the distance from the set $V$ to the origin is positive.

\begin{defin}
Let $\epsilon\in Uq^{-\N}$ and $\b\in\N$. We denote $E_{\b,\epsilon,Vq^{\R_+}}$ the vector space of functions $v\in\mathcal{O}(Vq^{\R_+})$ such that 
$$\left\|v(\tau)\right\|_{\b,\epsilon,Vq^{\R_+}}:=\sup_{\tau\in Vq^{\R_+}}\left\{\frac{|v(\tau)|}{e^{M\log^2\left|\frac{\tau}{\epsilon}\right|}}\left|\frac{\tau}{\epsilon}\right|^{-C\beta}\right\}|q|^{A_{1}\b^2}$$
is finite. 

Let $\delta>0$. We denote by $H(\epsilon,\delta,Vq^{\R_+})$ the complex vector space of all formal series $v(\tau,z)=\sum_{\beta\ge 0}v_{\b}(\tau)z^{\b}/\b!$ belonging to $\mathcal{O}(Vq^{\R_+})[[z]]$ such that 
$$\left\|v(\tau,z)\right\|_{(\epsilon,\delta,Vq^{\R_+})}:=\sum_{\b\ge 0}\left\|v_{\b}(\tau)\right\|_{\b,\epsilon,Vq^{\R_+}}\frac{\delta^\b}{\b!}<\infty.$$
It is straightforward to check that the pair $(H(\epsilon,\delta,Vq^{\R_+}),\left\|\cdot\right\|_{(\epsilon,\delta,Vq^{\R_+})})$ is a Banach space.
\end{defin}

We consider the formal integration operator $\partial_{z}^{-1}$ defined on $\mathcal{O}(Vq^{\R_+})[[z]]$ by
$$\partial_{z}^{-1}(v(\tau,z)):=\sum_{\b\ge1}v_{\beta-1}(\tau)\frac{z^\b}{\b!}\in\mathcal{O}(Vq^{\R_+})[[z]].$$

\begin{lemma}
\label{lema1}
Let $s,k,m_{1},m_{2}\in\N$, $\delta>0$, $\epsilon\in Uq^{-\N}$. We assume that the following conditions hold:
\begin{equation}\label{e79}
  m_1\le C(k+s)\quad,\quad m_2\ge 2(k+s)A_1.
\end{equation}
  
Then, there exists a constant $C_{1}=C_{1}(s,k,m_{1},m_{2},V,U,C,A_{1})$ (not depending on $\epsilon$ nor $\delta$) such that 
\begin{equation}\label{e84}
\left\|z^{s}\left(\frac{\tau}{\epsilon}\right)^{m_{1}}\partial^{-k}_{z}v(\tau,zq^{-m_{2}})\right\|_{(\epsilon,\delta,Vq^{\R_+})}\le C_{1}\delta^{k+s}\left\|v(\tau,z)\right\|_{(\epsilon,\delta,Vq^{\R_+})},
\end{equation}
for every $v\in H(\epsilon,\delta,Vq^{\R_+})$.
\end{lemma}
\begin{proof}
Let $v(\tau,z)=\sum_{\b\ge0}v_{\b}(\tau)\frac{z^\b}{\b!}\in\mathcal{O}(Vq^{\R_+})[[z]]$. We have that
\begin{align}
\left\|z^{s}\left(\frac{\tau}{\epsilon}\right)^{m_{1}}\partial^{-k}_{z}v(\tau,zq^{-m_{2}})\right\|_{(\epsilon,\delta,Vq^{\R_+})}&=\left\|\sum_{\b\ge k+s}\left(\frac{\tau}{\epsilon}\right)^{m_{1}}v_{\b-(k+s)}(\tau)\frac{\b!}{(\b-s)!}\frac{1}{q^{m_{2}(\b-s)}}\frac{z^{\b}}{\b!}\right\|_{(\epsilon,\delta,Vq^{\R_+})}\nonumber\\
&=\sum_{\b\ge k+s}\left\|\left(\frac{\tau}{\epsilon}\right)^{m_{1}}v_{\b-(k+s)}(\tau)\frac{\b!}{(\b-s)!}\frac{1}{q^{m_{2}(\b-s)}}\right\|_{\b,\epsilon,Vq^{\R_+}}\frac{\delta^{\b}}{\b!}\label{e1}
\end{align}
Taking into account the definition of the norm $\left\|\cdot\right\|_{\b,\epsilon,Vq^{\R_+}}$, we get
\begin{align}
&\left\|\left(\frac{\tau}{\epsilon}\right)^{m_{1}}v_{\b-(k+s)}(\tau)\frac{\b!}{(\b-s)!}\frac{1}{q^{m_{2}(\b-s)}}\right\|_{\b,\epsilon,Vq^{\R_+}}=\frac{\b!}{(\b-s)!}|q|^{A_{1}(\b-(k+s))^2}|q|^{p(\b)}\nonumber\\
&\sup_{\tau\in Vq^{\R_+}}\left\{\frac{|v_{\b-(k+s)}(\tau)|}{e^{M\log^{2}\left|\frac{\tau}{\epsilon}\right|}}\left|\frac{\tau}{\epsilon}\right|^{-C(\b-(k+s))}\left|\frac{\epsilon}{\tau}\right|^{C(k+s)-m_1}\right\},\label{e97}
\end{align}
with $p(\b)=A_{1}\b^2-A_{1}(\b-(k+s))^2-m_2(\b-s)$.
From (\ref{e79}) we derive $|\epsilon/\tau|^{C(k+s)-m_1}\le (C_U/C_V)^{C(k+s)-m_1}$ for every $\epsilon\in Uq^{-\N}$ and $\tau\in Vq^{\R_+}$, where $0<C_{V}:=\min\{|\tau|:\tau\in V\}$ and $0<C_U:=\max\{|\epsilon|:\epsilon\in U\}$. Moreover, 
$$p(\b)=(2(k+s)A_1-m_2)\b-(k+s)^2A_1+m_2s,$$
 for every $\b\in\N$. Regarding condition (\ref{e79}) we obtain the existence of $C_1>0$ such that
\begin{equation}
\label{e94}
 \left|\frac{\epsilon}{\tau}\right|^{C(k+s)-m_1}|q|^{p(\b)}\le C_1,
\end{equation}
for every $\tau\in Vq^{\R_+}$ and $\b\in\N$. Inequality (\ref{e84}) follows from (\ref{e1}), (\ref{e97}) and (\ref{e94}):
\begin{align*}
\left\|z^{s}\left(\frac{\tau}{\epsilon}\right)^{m_{1}}\partial^{-k}_{z}v(\tau,zq^{-m_{2}})\right\|_{(\epsilon,\delta,Vq^{\R_+})}&\le C_1\sum_{\b\ge k+s}\left\|v_{\b-(k+s)}(\tau)\right\|_{\b-(k+s),\epsilon,Vq^{\R_+}}\frac{\b!}{(\b-s)!} \frac{\delta^{\b}}{\b!}\\
&\le C_1 \delta^{k+s} \sum_{\b\ge k+s}\left\|v_{\b-(k+s)}(\tau)\right\|_{\b-(k+s),\epsilon,Vq^{\R_+}} \frac{\delta^{\b-(k+s)}}{(\b-(k+s))!}. 
\end{align*}
\end{proof}
\begin{lemma}
\label{lema2}
Let $F(\epsilon,\tau)$ be a holomorphic and bounded function defined on $Uq^{-\N}\times Vq^{\R_+}$. Then, there exists a constant $C_{2}=C_{2}(F,U,V)>0$ such that $$\left\|F(\epsilon,\tau)v_{\epsilon}(\tau,z)\right\|_{(\epsilon,\delta,Vq^{\R_+})}\le C_{2} \left\|v_{\epsilon}(\tau,z)\right\|_{(\epsilon,\delta,Vq^{\R_+})}$$
for every $\epsilon\in Uq^{-\N}$, every $\delta>0$ and all $v_{\epsilon}\in H(\epsilon,\delta, Vq^{\R_+})$. 
\end{lemma}
\begin{proof}
Direct calculations regarding the definition of the elements in $H(\epsilon,\delta, Vq^{\R_+})$ allow us to conclude when taking $C_{2}:=\max\{|F(\epsilon,\tau)|:\epsilon\in Uq^{-\N},\tau\in Vq^{\R_+}\}$.
\end{proof}

Let $S\ge1$ be an integer. For all $0\le k\le S-1$, let $m_{0,k},m_{1,k}$ be positive integers and $b_{k}(\epsilon,z)=\sum_{s\in I_k}b_{ks}(\epsilon)z^s$ be a polynomial in $z$, where $I_{k}$ is a finite subset of $\N$ and $b_{ks}(\epsilon)$ are holomorphic bounded functions on $D(0,r_0)$. We assume $\overline{Uq^{-\N}}\subseteq D(0,r_0))$.

We consider the following functional equation
\begin{equation}\label{fe1}
\partial_{z}^{S}W(\epsilon,\tau,z)=\sum_{k=0}^{S-1}\frac{b_{k}(\epsilon,z)}{(\tau+1)\epsilon^{m_{0,k}}}\tau^{m_{0,k}}(\partial_{z}^{k}W)(\epsilon,\tau,zq^{-m_{1,k}})
\end{equation}
with initial conditions
\begin{equation}
\label{ic1}
(\partial_z^jW)(\epsilon,\tau,0)=W_{j}(\epsilon,\tau)\quad,\quad 0\le j\le S-1,
\end{equation}
where the functions $(\epsilon,\tau)\mapsto W_{j}(\epsilon,\tau)$ belong to $\mathcal{O}(Uq^{-\N}\times Vq^{\R_+})$ for every $0\le j\le S-1$.

We make the following 

\textbf{Assumption (A)} For every $0\le k\le S-1$ and $s\in I_{k}$, we have
$$ m_{0,k}\le C(S-k+s)\quad,\quad m_{1,k}\ge 2(S-k+s)A_{1}.$$

\begin{theo}\label{teorema1}
Let Assumption (A) be fulfilled. We also make the following assumption on the initial conditions in (\ref{ic1}): there exist a constant $\Delta>0$ and $0<\tilde{M}<M$ such that for every $0\le j\le S-1$
\begin{equation}\label{e188}
|W_{j}(\epsilon,\tau)|\le \Delta e^{\tilde{M}\log^{2}\left|\frac{\tau}{\epsilon}\right|},
\end{equation}
for all $\tau\in Vq^{\R_+}$, $\epsilon\in Uq^{-\N}$. Then, there exists $W(\epsilon,\tau,z)\in\mathcal{O}(Uq^{-\N}\times Vq^{\R_+})[[z]]$ solution of (\ref{fe1})+(\ref{ic1}) such that if $W(\epsilon,\tau,z)=\sum_{\b\ge0}W_{\b}(\epsilon,\tau)\frac{z^{\b}}{\b!}$, then there exist $C_{2}>0$ and $0<\delta<1$ such that
\begin{equation}\label{e178}
|W_{\b}(\epsilon,\tau)|\le C_{2}\b! \left(\frac{|q|^{2A_{1}S}}{\delta}\right)^{\b}\left|\frac{\tau}{\epsilon}\right|^{C\b}e^{M\log^{2}\left|\frac{\tau}{\epsilon}\right|}|q|^{-A_{1}\b^2},\qquad\b\ge0
\end{equation}
for every $\epsilon\in Uq^{-\N}$ and $\tau\in Vq^{\R_+}$. 

\end{theo}
\begin{proof}
Let $\epsilon\in Uq^{-\N}$. We define the map $\mathcal{A}_{\epsilon}$ from $\mathcal{O}(Vq^{\R_+})[[z]]$ into itself by
\begin{equation}\label{e186}
\mathcal{A}_{\epsilon}(\tilde{W}(\tau,z)):=\sum_{k=0}^{S-1}\frac{b_{k}(\epsilon,z)}{(\tau+1)\epsilon^{m_{0,k}}}\tau^{m_{0,k}}\Big[(\partial_{z}^{k-S}\tilde{W})(\tau,zq^{-m_{1,k}})+\partial_{z}^{k}w_{\epsilon}(\tau,zq^{-m_{1,k}})\Big],
\end{equation}
where $w_{\epsilon}(\tau,z):=\sum_{j=0}^{S-1}W_{j}(\epsilon,\tau)\frac{z^{j}}{j!}$.
In the following lemma, we show the restriction of $\mathcal{A}_{\epsilon}$ to a neighborhood of the origin in $H(\epsilon,\delta,Vq^{\R_+})$ is a Lipschitz shrinking map for an appropriate choice of $\delta>0$.
\begin{lemma}
\label{lema3}
There exist $R>0$ and $\delta>0$ (not depending on $\epsilon$) such that:
\begin{enumerate}
\item[1.]  $\left\|\mathcal{A}_{\epsilon}(\tilde{W}(\tau,z))\right\|_{(\epsilon,\delta,Vq^{\R_+})}\le R$ for every $\tilde{W}(\tau,z)\in B(0,R)$. $B(0,R)$ denotes the closed ball centered at 0 with radius $R$ in $H(\epsilon,\delta,Vq^{\R_+})$.
\item[2.] 
$$\left\|\mathcal{A}_{\epsilon}(\tilde{W}_{1}(\tau,z))-\mathcal{A}_{\epsilon}(\tilde{W}_{2}(\tau,z))\right\|_{(\epsilon,\delta,Vq^{\R_+})}\le \frac{1}{2} \left\|\tilde{W}_{1}(\tau,z)-\tilde{W}_{2}(\tau,z)\right\|_{(\epsilon,\delta,Vq^{\R_+})}$$
 for every $\tilde{W}_{1},\tilde{W}_{2}\in B(0,R)$.
\end{enumerate}
\end{lemma}
\begin{proof}
Let $R>0$ and $0<\delta<1$.

For the first part we consider $\tilde{W}(\tau,z)\in B(0,R)\subseteq H(\epsilon,\delta, Vq^{\R_+})$. Lemma~\ref{lema1} and Lemma~\ref{lema2} can be applied so that
$$\left\|\mathcal{A}_{\epsilon}(\tilde{W}(\tau,z))\right\|_{(\epsilon,\delta, Vq^{\R_+})}$$
\begin{equation}\label{e202}
\le \sum_{k=0}^{S-1}\sum_{s\in I_{k}}\frac{M_{ks}}{M_{1}}\Big[C_{1}\delta^{S-k+s}\left\|\tilde{W}(\tau,z)\right\|_{(\epsilon,\delta, Vq^{\R_+})}+\left\|z^s\left(\frac{\tau}{\epsilon}\right)^{m_{0,k}}\partial_{z}^{k}w_{\epsilon}(\tau,zq^{-m_{1,k}})\right\|_{(\epsilon,\delta,Vq^{\R_+})}\Big], 
\end{equation}
with $M_{ks}=\sup_{\epsilon\in Uq^{-\N}}|b_{ks}(\epsilon)|<\infty$, $s\in I_{k}$, $0\le k\le S-1$. Taking into account the definition of $H(\epsilon,\delta,Vq^{\R_+})$ and (\ref{e188}) we have
$$\left\|z^s\Big(\frac{\tau}{\epsilon}\Big)^{m_{0,k}}\partial^{k}_{z}w_{\epsilon}(\tau,zq^{-m_{1,k}})\right\|_{(\epsilon,\delta,Vq^{\R_+})}=\left\|\sum_{j=0}^{S-1-k}\Big(\frac{\tau}{\epsilon}\Big)^{m_{0,k}}W_{j+k}(\epsilon,\tau)\frac{z^{j+s}}{j!q^{m_{1,k}j}}\right\|_{(\epsilon,\delta,Vq^{\R_+})}$$
\begin{align}
&=\sum_{j=0}^{S-1-k}\sup_{\tau\in Vq^{\R_+}}\left\{\frac{|W_{j+k}(\epsilon,\tau)|}{e^{M\log^{2}\left|\frac{\tau}{\epsilon}\right|}}\left|\frac{\tau}{\epsilon}\right|^{m_{0,k}-C(j+s)}\right\}|q|^{A_{1}(j+s)^2}\frac{\delta^{j+s}}{j!|q|^{m_{1,k}j}}\nonumber\\
&\le \Delta\sum_{j=0}^{S-1-k} \frac{|q|^{A_{1}(j+s)^2}\delta^{j+s}}{j!|q|^{m_{1,k}j}}\max\{e^{-(M-\tilde{M})\log^2(x)}x^{m_{0,k}-C(j+s)}:x>0,0\le j+k\le S-1,s\in I_{k} \}\label{e210}\\
&\le \Delta C'_{2}\nonumber,
\end{align}
for a positive constant $C'_{2}$.

We conclude this first part from an appropriate choice of $R$ and $\delta>0$.

For the second part we take $\tilde{W}_{1},\tilde{W}_{2}\in B(0,R)\subseteq H(\epsilon,\delta, Vq^{\R_+})$. Similar arguments as before yield
$$\left\|\mathcal{A}_{\epsilon}(\tilde{W}_{1})-\mathcal{A}_{\epsilon}(\tilde{W}_{2})\right\|_{(\epsilon,\delta, Vq^{\R_+})}\le \sum_{k=0}^{S-1}\sum_{s\in I_{k}}\frac{M_{ks}}{M_{1}}C_{1}\delta^{S-k+s}\left\|\tilde{W}_{1}- \tilde{W}_2\right\|_{(\epsilon,\delta, Vq^{\R_+})}.$$
An adequate choice for $\delta>0$ allows us to conclude the proof.
\end{proof}

We choose constants $R,\delta$ as in the previous lemma.\\
From Lemma~\ref{lema3} and taking into account the shrinking map theorem on complete metric spaces, we guarantee the existence of $\tilde{W}_{\epsilon}(\tau,z)\in H(\epsilon,\delta, Vq^{\R_+})$ which is a fixed point for $\mathcal{A}_{\epsilon}$ in $B(0,R)$, it is to say, $\big\|\tilde{W}_{\epsilon}(\tau,z)\big\|_{(\epsilon,\delta,Vq^{\R_+})}\le R$ and $\mathcal{A}_{\epsilon}(\tilde{W}_{\epsilon}(\tau,z))=\tilde{W}_{\epsilon}(\tau,z)$. 

Let us define
\begin{equation}\label{ee239} W_{\epsilon}(\tau,z):=\partial_{z}^{-S}\tilde{W}_{\epsilon}(\tau,z)+w_{\epsilon}(\tau,z).
\end{equation}
If we write $\tilde{W}_{\epsilon}(\tau,z)=\sum_{\b\ge 0}\tilde{W}_{\b,\epsilon}(\tau)\frac{z^{\b}}{\b!}$ and $W_{\epsilon}(\tau,z)=\sum_{\b\ge 0}W_{\b,\epsilon}(\tau)\frac{z^{\b}}{\b!}$, then we have that $W_{\b+S,\epsilon}\equiv\tilde{W}_{\beta,\epsilon}$ for $\beta\ge 0$ and $W_{j,\epsilon}(\tau)=W_{j}(\epsilon,\tau)$, $0\le j\le S-1$.

From $\big\|\tilde{W}_{\epsilon}(\tau,z)\big\|_{(\epsilon,\delta,Vq^{\R_+})}\le R$ we arrive at $\big\|\tilde{W}_{\b,\epsilon}\big\|_{\b,\epsilon,Vq^{\R_+}}\le R\beta!\left(\frac{1}{\delta}\right)^{\beta}$
for every $\b\ge0$. This implies
$$|\tilde{W}_{\b,\epsilon}(\tau)|\le R\b! \left(\frac{1}{\delta}\right)^{\beta}\left|\frac{\tau}{\epsilon}\right|^{C\b} e^{M\log^{2}\left|\frac{\tau}{\epsilon}\right|}|q|^{-A_{1}\b^2}, $$
for every $\b\ge0$ and $\tau\in Vq^{\R_+}$. 

This is valid for every $\epsilon\in Uq^{-\N}$. We define $W(\epsilon,\tau,z):=W_{\epsilon}(\tau,z)$ and $W_{\b}(\epsilon,\tau):=W_{\b,\epsilon}(\tau)$ for every $(\epsilon,\tau)\in Uq^{-\N}\times Vq^{\R_{+}}$, $z\in\C$ and $\b\ge S$. From (\ref{ee239}), it is straightforward to prove that $W(\epsilon,\tau,z)=\sum_{\b\ge0}W_{\b}(\epsilon,\tau)\frac{z^{\b}}{\b!}$ is a solution of (\ref{fe1})+(\ref{ic1}).

Moreover, holomorphy of $W_{\b}$ in $Uq^{-\N}\times Vq^{\R_{+}}$ for every $\b\ge0$ can be deduced from the recursion formula verified by the coefficients:
\begin{equation}\label{e223}
\frac{W_{h+S}(\epsilon,\tau)}{h!}=\sum_{k=0}^{S-1}\sum_{h_{1}+h_{2}=h,h_{1}\in I_{k}}\frac{b_{kh_{1}}(\epsilon)\tau^{m_{0,k}}}{(\tau+1)\epsilon^{m_{0,k}}}\frac{W_{h_{2}+k}(\epsilon,\tau)}{h_{2}!q^{m_{1,k}h_{2}}},\quad h\ge0.
\end{equation}
This implies $W_{\b}(\epsilon,\tau)$ is holomorphic in $Uq^{-\N}\times Vq^{\R_{+}}$ for every $\b\in\N$.

It only rests to prove (\ref{e178}). Upper and lower bounds for the modulus of the elements in $Uq^{-\N}$ and $Vq^{\R_+}$ respectively and usual calculations lead us to assure the existence of a positive constant $R_{1}>0$ such that
\begin{equation}
\label{e193}
 |W_{\beta}(\epsilon,\tau)|=|\tilde{W}_{\beta-S}(\epsilon,\tau)|\le R_{1}\b! \left(\frac{|q|^{2A_{1}S}}{\delta}\right)^{\b}\left|\frac{\tau}{\epsilon}\right|^{C\b}e^{M\log^{2}\left|\frac{\tau}{\epsilon}\right|}|q|^{-A_{1}\b^2},
 \end{equation}  
for every $\b\ge S$, and for every $\epsilon\in Uq^{-\N}$ and $\tau\in Vq^{\R_+}$. This concludes the proof for $\b\ge S$.

Hypothesis (\ref{e188}) leads us to obtain (\ref{e193}) for $0\le k\le S-1$.
\end{proof}
\textit{Remark:}
If $s>0$ for every $s\in I_{k}$, $0\le k\le S-1$ then any choice of $R>0$ is valid for a small enough $\delta>0$ in order to achieve the result.
\end{section}

\begin{section}{Second Cauchy problem in a weighted Banach space of Taylor series}\label{seccion2}

This section is devoted to the study of the same equation as in the previous section when the initial conditions are of a different nature. Proofs will only be scketched not to repeat calculations.

Let $1<\rho_0$ and $U\subseteq\C^{\star}$ a bounded and open set with positive distance to the origin. $D(0,\rho_0)\setminus \{0\}$ will be denoted $\dot{D}_{\rho_0}$ in this section.
$M,A_{1},C$ remain the same positive constants as in the previous section.
 
\begin{defin} 
Let $r_{0}>0$, $\epsilon\in D(0,r_{0})\setminus\{0\}$ and $\b\in\N$. We denote $E^{2}_{\b,\epsilon,\dot{D}_{\rho_0}}$ the vector space of functions $v\in\mathcal{O}(\dot{D}_{\rho_0})$ such that 
$$|v(\tau)|_{\b,\epsilon,\dot{D}_{\rho_0}}:=\sup_{\tau\in \dot{D}_{\rho_0}}\big\{|v(\tau)|\frac{|\epsilon|^{C\beta}}{e^{M\log^2|\epsilon/\tau|}}\big\}|q|^{A_{1}\b^2},$$
is finite.
Let $\delta>0$. We denote by $H_2(\epsilon,\delta,\dot{D}_{\rho_0})$ the vector space of all formal series $v(\tau,z)=\sum_{\beta\ge 0}v_{\b}(\tau)z^{\b}/\b!$ belonging to $\mathcal{O}(\dot{D}_{\rho_0})[[z]]$ such that 
$$|v(\tau,z)|_{(\epsilon,\delta,\dot{D}_{\rho_0})}:=\sum_{\b\ge 0}|v_{\b}(\tau)|_{\b,\epsilon,\dot{D}_{\rho_0}}\frac{\delta^\b}{\b!}<\infty.$$
It is straightforward to check that the pair $(H_2(\epsilon,\delta,\dot{D}_{\rho_0}),|\cdot|_{(\epsilon,\delta,\dot{D}_{\rho_0})})$ is a Banach space.
\end{defin}
\begin{lemma}
\label{lema4}
Let $s,k,m_{1},m_{2}\in\N$, $\delta>0$ and $\epsilon\in D(0,r_{0})\setminus\{0\}$. We assume that the following conditions hold:
\begin{equation}\label{e79sub2}
  m_1\le C(k+s)\quad,\quad m_2\ge 2(k+s)A_1.
\end{equation}
  
Then, there exists a constant $C_{1}=C_{1}(s,k,m_{1},m_{2},\dot{D}_{\rho_0},U)$ (not depending on $\epsilon$ nor $\delta$) such that 
\begin{equation}\label{e84sub2}
\left|z^{s}\left(\frac{\tau}{\epsilon}\right)^{m_{1}}\partial^{-k}_{z}v(\tau,zq^{-m_{2}})\right|_{(\epsilon,\delta,\dot{D}_{\rho_0})}\le C_{1}\delta^{k+s}\left|v(\tau,z)\right|_{(\epsilon,\delta,\dot{D}_{\rho_0})},
\end{equation}
for every $v\in H_2(\epsilon,\delta,\dot{D}_{\rho_0})$.
\end{lemma}
\begin{proof}
Let $v(\tau,z)\in\mathcal{O}(\dot{D}_{\rho_0})[[z]]$. The proof follows similar steps as in Lemma~\ref{lema1}. We have 
$$
\left|z^{s}\left(\frac{\tau}{\epsilon}\right)^{m_{1}}\partial^{-k}_{z}v(\tau,zq^{-m_{2}})\right|_{(\epsilon,\delta,\dot{D}_{\rho_0})}=\sum_{\b\ge k+s}\left|\left(\frac{\tau}{\epsilon}\right)^{m_{1}}v_{\b-(k+s)}(\tau)\frac{\b!}{(\b-s)!}\frac{1}{q^{m_{2}(\b-s)}}\right|_{\b,\epsilon,\dot{D}_{\rho_0}}\frac{\delta^{\b}}{\b!}.$$
From the definition of the norm $|\cdot|_{\b,\epsilon,\dot{D}_{\rho_0}}$, we get
$$\left|\left(\frac{\tau}{\epsilon}\right)^{m_{1}}v_{\b-(k+s)}(\tau)\frac{\b!}{(\b-s)!}\frac{1}{q^{m_{2}(\b-s)}}\right|_{\b,\epsilon,\dot{D}_{\rho_0}}\le\frac{\b!}{(\b-s)!}|q|^{A_{1}(\b-(k+s))^2}|q|^{p(\b)}$$
$$\times\sup_{\tau\in \dot{D}_{\rho_0}}\left\{\frac{|v_{\b-(k+s)}(\tau)|}{e^{M\log^{2}|\epsilon/\tau|}}|\epsilon|^{C(\b-(k+s))}\right\}\rho_0^{m_1}|\epsilon|^{C(k+s)-m_1},$$
with $p(\b)=A_{1}\b^2-A_{1}(\b-(k+s))^2-m_2(\b-s)$.
Identical arguments as in Lemma~\ref{lema1} allow us to conclude.
\end{proof}
\begin{lemma}
\label{lema5}
Let $0<\rho_{1}<1$ and $F(\epsilon,\tau)$ a holomorphic and bounded function defined on $D(0,r_{0})\times \C\setminus D(0,\rho_1) $. Then, there exists a constant $C_{2}=C_{2}(F)>0$ such that $$\left|F(\epsilon,\tau)v_{\epsilon}(\tau,z)\right|_{(\epsilon,\delta,\dot{D}_{\rho_0})}\le C_{2} \left|v_{\epsilon}(\tau,z)\right|_{(\epsilon,\delta,\dot{D}_{\rho_0})}$$
 for every $\epsilon\in D(0,r_{0})\setminus\{0\}$, every $\delta>0$ and every $v_{\epsilon}\in H_2(\epsilon,\delta, \dot{D}_{\rho_0})$. 
\end{lemma}

Let $S,r_{0},m_{0,k},m_{1,k}$ and $b_{k}$ as in Section~\ref{seccion1} and $\rho_0>0$.  We consider the Cauchy problem 
\begin{equation}\label{fe2}
\partial_{z}^{S}W(\epsilon,\tau,z)=\sum_{k=0}^{S-1}\frac{b_{k}(\epsilon,z)}{(\tau+1)\epsilon^{m_{0,k}}}\tau^{m_{0,k}}(\partial_{z}^{k}W)(\epsilon,\tau,zq^{-m_{1,k}})
\end{equation}
with initial conditions
\begin{equation}
\label{ic2}
(\partial_z^jW)(\epsilon,\tau,0)=W_{j}(\epsilon,\tau)\quad,\quad 0\le j\le S-1,
\end{equation}
where the functions $(\epsilon,\tau)\mapsto W_{j}(\epsilon,\tau)$ belong to $\mathcal{O}((D(0,r_{0})\setminus\{0\})\times \dot{D}_{\rho_0})$ for every $0\le j\le S-1$.
\begin{theo}\label{teorema2}
Let Assumption (A) be fulfilled. We make the following assumption on the initial conditions (\ref{ic2}): there exist constants $\Delta>0$ and $0<\tilde{M}<M$ such that
\begin{equation}\label{e315}
|W_{j}(\epsilon,\tau)|\le \Delta e^{\tilde{M}\log^{2}\left|\frac{\epsilon}{\tau}\right|},
\end{equation}
for every $\tau\in \dot{D}_{\rho_0}$, $\epsilon\in D(0,r_{0})\setminus\{0\}$ and $0\le j\le S-1$. Then, there exists $W(\epsilon,\tau,z)\in\mathcal{O}((D(0,r_{0})\setminus\{0\})\times \dot{D}_{\rho_0})[[z]]$ solution of (\ref{fe2})+(\ref{ic2}) such that if $W(\epsilon,\tau,z)=\sum_{\b\ge0}W_{\b}(\epsilon,\tau)\frac{z^{\b}}{\b!}$, then there exist $C_{3}>0$ and $0<\delta<1$ such that
\begin{equation}\label{e325}
|W_{\b}(\epsilon,\tau)|\le C_{3}\b! \left(\frac{|q|^{2A_{1}S}}{\delta}\right)^{\b}|\epsilon|^{-C\b}e^{M\log^{2}\left|\frac{\epsilon}{\tau}\right|}|q|^{-A_{1}\b^2},\qquad \b\ge 0,
\end{equation}
for every $\epsilon\in D(0,r_{0})\setminus\{0\}$ and $\tau\in \dot{D}_{\rho_0}$.
\end{theo}
\begin{proof}
The proof of Theorem~\ref{teorema1} can be adapted here so details will be omitted.

Let $\epsilon\in D(0,r_{0})\setminus\{0\}$ and $0<\delta<1$. We consider the map $\mathcal{A}_{\epsilon}$ from $\mathcal{O}(\dot{D}_{\rho_0})[[z]]$ into itself defined as in (\ref{e186}) and construct $w_{\epsilon}(\tau,z)$ as above. 
From (\ref{e315}) we derive
$$\left|z^s\Big(\frac{\tau}{\epsilon}\Big)^{m_{0,k}}\partial^{k}_{z}w_{\epsilon}(\tau,zq^{-m_{1,k}})\right|_{(\epsilon,\delta,\dot{D}_{\rho_0})}$$
\begin{align}
&=\sum_{j=0}^{S-1-k}\sup_{\tau\in \dot{D}_{\rho_0}}|W_{j+k}(\epsilon,\tau)|\frac{|\epsilon|^{C(j+s)}}{e^{M\log^{2}\left|\frac{\epsilon}{\tau}\right|}}\Big|\frac{\tau}{\epsilon}\Big|^{m_{0,k}}|q|^{A_{1}(j+s)^2}\frac{\delta^{j+s}}{j!|q|^{m_{1,k}j}}\nonumber\\
&\le \Delta  C'_{3}\label{e335},
\end{align}
for a positive constant $C'_{3}$ not depending on $\epsilon$ nor $\delta$.

Lemma~\ref{lema4}, Lemma~\ref{lema5} and (\ref{e335}) allow us to affirm that one can find $R>0$ and $\delta>0$ such that the restriction of $\mathcal{A}_{\epsilon}$ to the disc $D(0,R)$ in $H_2(\epsilon,\delta, \dot{D}_{\rho_0})$ is a Lipschitz shrinking map for appropriate choices of $R>0$ and $\delta>0$. Moreover, there exists $\tilde{W}_{\epsilon}(\tau,z)\in H_{2}(\epsilon,\delta,\dot{D}_{\rho_0})$ which is a fixed point for $\mathcal{A}_{\epsilon}$ in $B(0,R)$.

If we put $\tilde{W}_{\epsilon}(\tau,z)=\sum_{\b\ge 0}\tilde{W}_{\b,\epsilon}(\tau)\frac{z^{\b}}{\b!}$, then one gets $|\tilde{W}_{\b,\epsilon}|_{\b,\epsilon,\dot{D}_{\rho_0}}\le R\beta!\left(\frac{1}{\delta}\right)^{\beta}$ for $\b\ge0$. This implies
$$|\tilde{W}_{\b,\epsilon}(\tau)|\le R\b! \left(\frac{1}{\delta}\right)^{\beta}|\epsilon|^{-C\b} e^{M\log^{2}\left|\frac{\epsilon}{\tau}\right|}|q|^{-A_{1}\b^2},\quad \b\ge 0,\tau\in \dot{D}_{\rho_0}. $$
The formal power series $$W(\epsilon,\tau,z):=\sum_{\b\ge S}\tilde{W}_{\b-S,\epsilon}(\tau)\frac{z^{\b}}{\b!}+w_{\epsilon}(\tau,z):=\sum_{\b\ge0}W_{\b}(\epsilon,\tau)\frac{z^{\b}}{\b!}$$ turns out to be a solution of (\ref{fe2})+(\ref{ic2}) verifying that $W_{\b}(\epsilon,\tau)$ is a holomorphic function in $D(0,r_{0})\setminus\{0\}\times \dot{D}_{\rho_0}$ and the estimates (\ref{e325}) hold for $\b\ge0$.

\end{proof}

\end{section}

\begin{section}{Analytic solutions in a small parameter of a singularly perturbed problem}

\begin{subsection}{A $q-$analog of the Laplace transform and $q-$asymptotic expansion}
In this subsection, we recall the definition and several results related to Theta Jacobi function and also a $q-$analog of the Laplace transform which was firstly developed by  C. Zhang in~\cite{zhang1}. 

Let $q\in\C$ such that $|q|>1$. 

Theta Jacobi function is defined in $\C^{\star}$ by
$$\Theta(x)=\sum_{n\in\Z}q^{-n(n-1)/2}x^n,\quad x\in\C^{\star}.$$
From the fact that Theta Jacobi function satisfies the functional equation $xq\Theta(x)=\Theta(qx)$, for $x\neq0$, we have
 \begin{equation}\label{e231}\Theta(q^mx)=q^{\frac{m(m+1)}{2}}x^m\Theta(x),\quad x\in\C,x\neq0\end{equation} for every $m\in\Z$.
The following lower bounds for Theta Jacobi function will be useful in the sequel.
\begin{lemma}\label{lema374}
Let $\delta>0$. We have 
\begin{equation}\label{e374}
|\Theta(x)|\ge \delta e^{\frac{\log^{2}|x|}{2\log|q|}}p(|x|), 
\end{equation} 
for every $x\in\C^{\star}$ such that $|1+xq^{k}|>\delta$ for all $k\in\Z$. Here $p(|x|)$ is a finite linear combination of elements in $\{|x|^{m}:m\in\mathbb{Q}\}$ with positive coefficients. 
\end{lemma}
\begin{proof}
Let $\delta>0$. From Lemma 5.1.6 in \cite{ramissauloyzhang} we get the existence of a positive constant $C_1$ such that $|\Theta(x)|\ge C_{1}\delta\Theta_{|q|}(|x|)$ for every $x\in\C^{\star}$ such that $|1+xq^{k}|>\delta$ for all $k\in\Z$. Now,
$$\Theta_{|q|}(|x|)=\sum_{n\in\Z}|q|^{-\frac{n(n-1)}{2}}|x|^{n}\ge \max_{n\in\Z}|q|^{-\frac{n(n-1)}{2}}|x|^{n}.$$
Let us fix $|x|$. The function $f(t)=\exp(-1/2\log(|q|)t(t-1)+t\log|x|)$ takes its maximum value at $t_0=\frac{\log|x|}{\log|q|}+\frac{1}{2}$ with $f(t_0)=C_2\exp(\frac{\log^2|x|}{2\log|q|})|x|^{1/2}$, for certain $C_{2}>0$. Taking into account that $r-1<\left\lfloor r\right\rfloor\le r<\left\lfloor r\right\rfloor +1$ for $r\in\R$ ($\left\lfloor \cdot\right\rfloor$ stands for the entire part), we conclude from usual estimates.
\end{proof}
\begin{corol}\label{coro361}
Let $\delta>0$. For any $\xi\in(0,1)$ there exists $C_{\xi}=C_{\xi}(\delta)>0$ such that 
\begin{equation}\label{e376}
|\Theta(x)|\ge C_{\xi}e^{\frac{\xi\log^{2}|x|}{2\log|q|}},
\end{equation} 
for every $x\in\C^{\star}$ such that  $|1+xq^{k}|>\delta$, for all $k\in\Z$.
\end{corol}
From now on, $(\mathbb{H},\left\|\cdot\right\|_{\mathbb{H}})$ stands for a complex Banach space.

For any $\lambda\in\C$ and $\delta>0$ we denote 
$$\mathcal{R}_{\lambda,q,\delta}:=\{z\in\C^{\star}: |1+\frac{\lambda}{zq^{k}}|>\delta,\forall k\in\R\}.$$
The following definition corresponds to a $q-$analog of Laplace transform and can be found in~\cite{zhang1} when working with sectors in the complex plane.
\begin{prop}\label{prop11}
Let $\delta>0$ and $\rho_{0}>0$. We fix an open and bounded set $V$ in $\C^{\star}$ and $\rho>0$ such that $D(0,\rho_{0})\cap V\neq\emptyset$. Let $\lambda\in D(0,\rho_0)\cap V$ and $f$ be a holomorphic function defined in $\dot{D}_{\rho_0}$ with values in $\mathbb{H}$ such that can be extended to a function $F$ defined in $\dot{D}_{\rho_0}\cup Vq^{\R_+}$ and  
\begin{equation}\label{e3852}
\left\|F(x)\right\|_{\mathbb{H}}\le C_{1}e^{\overline{M}\log^{2}|x|},\qquad x\in Vq^{\R}:=Vq^{\R_+}\cup V(q^{-1})^{\R_{+}},
\end{equation}
for positive constants $C_{1}>0$ and $0<\overline{M}<\frac{1}{2\log|q|}$.

Let $\pi_{q}=\log(q)\prod_{n\ge 0}(1-q^{-n-1})^{-1}$ and put
\begin{equation}\label{e4010}
\mathcal{L}_{q;1}^{\lambda}F(z)=\frac{1}{\pi_{q}}\int_{0}^{\infty\lambda}\frac{F(\xi)}{\Theta(\frac{\xi}{z})}\frac{d\xi}{\xi},
\end{equation}
where the path $[0,\infty\lambda]$ is given by $t\in(-\infty,\infty)\mapsto q^{t}\lambda$.
Then, $\mathcal{L}_{q;1}^{\lambda}F$ defines a holomorphic function in $\mathcal{R}_{\lambda,q,\delta}$ and it is known as the $q-$Laplace transform of $f$ following direction $[\lambda]$.
\end{prop}
\begin{proof}

Let $K\subseteq\mathcal{R}_{\lambda,q,\delta}$ be a compact set and $z\in K$. From the parametrization of the path $[0,\infty\lambda]$ we have
$$\int_{0}^{\infty\lambda}\frac{F(\xi)}{\Theta\left(\frac{\xi}{z}\right)}\frac{d\xi}{\xi}=\log(q)\int_{-\infty}^{\infty}\frac{F(q^t\lambda)}{\Theta\left(\frac{q^t\lambda}{z}\right)}dt. $$
Let $0<\xi_{1}<1$ such that $0<\overline{M}<\frac{\xi_{1}}{2\log|q|}$ and let $t\in\R$. We have $w=\frac{q^t\lambda}{z}$ satisfies $|1+q^{k}w|>\delta$ for every $k\in\Z$. Corollary~\ref{coro361} and (\ref{e3852}) yields
$$\int_{-\infty}^{\infty}\left\|\frac{F(q^t\lambda)}{\Theta\left(\frac{q^t\lambda}{z}\right)}\right\|_{\mathbb{H}}dt\le\int_{-\infty}^{\infty}\frac{C_{1}e^{\overline{M}\log^2|q^t\lambda|}}{C_{\xi_{1}}e^{\frac{\xi_{1}}{2\log|q|}\log^2|q^t\lambda/z|}}dt\le L_{1}\int_{-\infty}^{\infty}|q^t\lambda|^{\frac{\xi_{1}L_{2}}{\log|q|}}e^{(\overline{M}-\frac{\xi_{1}}{2\log|q|})\log^{2}|q^t\lambda|}dt,$$
for positive constants $L_{1},L_{2}$. This integral is convergent and does not depend on $z\in K$.
\end{proof}
\textit{Remark:} If we let $\overline{M}=\frac{1}{2\log|q|}$, then $\mathcal{L}^{\lambda}_{q;1}F$ will only remain holomorphic in $\mathcal{R}_{\lambda,q,\delta}\cap D(0,r_{1})$ for certain $r_{1}>0$.

In the next proposition, we recall a commutation formula for the $q-$Laplace transform and the multiplication by a polynomial.
\begin{prop}\label{proposicion2}
Let $V$ be an open and bounded set in $\C^{\star}$ and $D(0,\rho_{0})$ such that $V\cap D(0,\rho_{0})\neq \emptyset$. Let $\phi$ a holomorphic function on $Vq^{\R_{+}}\cup \dot{D}_{\rho_0}$ with values in the Banach space $(\mathbb{H},\left\|\cdot\right\|_{\mathbb{H}})$ which satisfies the following estimates: there exist $C_{1}>0$ and $0<\overline{M}<\frac{1}{2\log|q|}$ such that 
\begin{equation}\label{e423}
\left\|\phi(x)\right\|_{\mathbb{H}}< C_{1}e^{\overline{M}\log^{2}|x|},\qquad x\in Vq^{\R_{+}}.
\end{equation}
Then, the function $m\phi(\tau)=\tau\phi(\tau)$ is holomorphic on $Vq^{\R_+}\cup \dot{D}_{\rho_0}$ and satisfies estimates in the shape above. Let $\lambda\in V\cap D(0,\rho_{0})$ and $\delta>0$. We have the following equality
$$\mathcal{L}_{q;1}^{\lambda}(m\phi)(t)=t\mathcal{L}^{\lambda}_{q;1}\phi(qt)$$
for every $t\in \mathcal{R}_{\lambda,q,\delta}$.
\end{prop}
\begin{proof}
It is direct to prove that $m\phi$ is a holomorphic function in $Vq^{\R_{+}}\cup \dot{D}_{\rho_0}$ and also that $m\phi$ verifies bounds as in (\ref{e423}). From (\ref{e231}) taking $m=-1$ we derive
\begin{align*}
\mathcal{L}_{q;1}^{\lambda}(m\phi)(t)&=\frac{1}{\pi_{q}}\int_{0}^{\infty\lambda}\frac{(m\phi)(\xi)}{\Theta(\frac{\xi}{t})}\frac{d\xi}{\xi}=\frac{1}{\pi_{q}}\int_{0}^{\infty\lambda}\frac{\phi(\xi)}{\Theta(\frac{\xi}{t})}d\xi\\
&=\frac{1}{\pi_{q}}\int_{0}^{\infty\lambda}\frac{\phi(\xi)}{\frac{\xi}{t}\Theta(\frac{\xi}{qt})}d\xi=t\mathcal{L}_{q;1}^{\lambda}(\phi)(qt),
\end{align*}
for every $t\in \mathcal{R}_{\lambda,q,\delta}$.
\end{proof}

\end{subsection}

\begin{subsection}{Analytic solutions in a parameter of a singularly perturbed Cauchy problem}
The following definition of a good covering firstly appeared in~\cite{ramissauloyzhang}, p.36.
\begin{defin}\label{def527}
Let $I=(I_{1},I_{2})$ be a pair of two open intervals in $\R$ each one of length smaller than $1/4$ and let $U_{I}$ be the corresponding open bounded set in $\C^{\star}$ defined by
$$U_{I}=\{e^{2\pi ui}q^v\in\C^{\star}: u\in I_{1},v\in I_{2}\}. $$
Let $\mathcal{I}$ be a finite family of tuple $I$ as above verifying
\begin{enumerate}
\item $\cup_{I\in\mathcal{I}}(U_{I}q^{-\N})=\nu\setminus\{0\}$, where $\nu$ is a neighborhood of 0 in $\C$, and
\item the open sets $U_{I}q^{-\N}$, $I\in\mathcal{I}$ are four by four disjoint.
\end{enumerate}
Then, we say $(U_{I}q^{-\N})_{I\in\mathcal{I}}$ is a good covering.
\end{defin}
\begin{defin}\label{definicion536}
Let $(U_{I}q^{-\N})_{I\in\mathcal{I}}$ be a good covering. Let $\delta>0$. We consider a family of open bounded sets $(V_{I})_{I\in\mathcal{I}},\mathcal{T}$ in $\C^{\star}$ such that:
\begin{enumerate}
\item There exists $1<\rho_0$ with $V_{I}\cap D(0,\rho_{0})\neq\emptyset$, for all $I\in\mathcal{I}$.
\item For every $I\in\mathcal{I}$ and $\tau\in V_{I}q^{\R}$, $|\tau+1|>\delta$.
\item For every $I\in\mathcal{I}$, $t\in\mathcal{T}$, $\epsilon_{u}\in U_{I}$ and $\lambda_{v}\in V_{I}\cap D(0,\rho_0)$, we have
$$|1+\frac{\lambda_{v}}{\epsilon_{u}tq^{r}}|>\delta,$$
for every $r\in\R$.
\item $|t|\le 1$ for every $t\in\mathcal{T}$.
\end{enumerate}
We say the family $\{(V_{I})_{I\in\mathcal{I}},\mathcal{T}\}$ is associated to the good covering $(U_{I}q^{-\N})_{I\in\mathcal{I}}$.
\end{defin}
Let $S\ge1$ be an integer. For every $0\le k\le S-1$, let $m_{0,k},m_{1,k}$ be positive integers and $b_{k}(\epsilon,z)=\sum_{s\in I_{k}}b_{ks}(\epsilon)z^s$ be a polynomial in $z$, where $	I_{k}$ is a subset of $\N$ and $b_{ks}(\epsilon)$ are bounded holomorphic functions on some disc $D(0,r_{0})$ in $\C$, $0<r_{0}\le 1$. Let $(U_{I}q^{-\N})_{I\in\mathcal{I}}$ be a good covering such that $U_{I}q^{-\N}\subseteq D(0,r_{0})$ for every $I\in \mathcal{I}$.

\textbf{Assumption (B)}: $$M\le \frac{1}{2\log|q|}.$$

\begin{defin}
Let $T_0>0$, $\rho_0>1$ such that $V\cap D(0,\rho_0)\neq \emptyset$. Let $\Delta,\tilde{M}>0$ such that $\tilde{M}<M$ and $(\epsilon,\tau)\mapsto W(\epsilon,\tau)$ a bounded holomorphic function on $D(0,r_0)\setminus\{0\}\times \dot{D}_{\rho_{0}}$ verifying 
$$|W(\epsilon,\tau)|\le \Delta e^{\tilde{M}\log^2|\tau/\epsilon|}, $$ 
for every $(\epsilon,\tau)\in D(0,r_0)\setminus\{0\}\times \dot{D}_{\rho_0}$. Assume moreover that $W(\epsilon,\tau)$ can be extended to an analytic function $(\epsilon,\tau)\mapsto W_{U V}(\epsilon,\tau)$ on $Uq^{-\N}\times (Vq^{\R_+}\cup \dot{D}_{\rho_0})$ and 
\begin{equation}\label{e658}
|W_{UV}(\epsilon,\tau)|\le \Delta e^{\tilde{M}\log^2|\tau/\epsilon|},
\end{equation} 
for every $(\epsilon,\tau)\in Uq^{-\N}\times (Vq^{\R_+}\cup \dot{D}_{\rho_0})$. We say that the set $\{W, W_{UV},\rho_0 \} $ is admissible.
\end{defin}

Let $\mathcal{I}$ be a finite family of indices. For every $I\in\mathcal{I}$, we consider the following singularly perturbed Cauchy problem
\begin{equation}\label{ee1}\epsilon t \partial_{z}^{S}X_I(\epsilon,qt,z)+\partial_z^{S}X_I(\epsilon,t,z)=\sum_{k=0}^{S-1}b_k(\epsilon,z)(t\sigma_q)^{m_{0,k}}(\partial_z^kX_I)(\epsilon,t,zq^{-m_{1,k}})
\end{equation}
with $b_{k}$ as in (\ref{fe1}), and with initial conditions
\begin{equation}\label{ii1} (\partial_{z}^{j}X_{I})(\epsilon,t,0)=\phi_{I,j}(\epsilon,t)\quad,\quad 0\le j\le S-1,
\end{equation}
where the functions $\phi_{I,j}(\epsilon,t)$ are constructed as follows. Let $\{(V_{I})_{I\in\mathcal{I}},\mathcal{T}\}$ be a family of open sets associated to the good covering $(U_{I}q^{-\N})_{I\in\mathcal{I}}$. For every $0\le j\le S-1$ and $I\in\mathcal{I}$, let $\{W_j,W_{U_I,V_I,j},\rho_0 \}$ be an admissible set. Let $\lambda_I$ be a complex number in $V_I\cap D(0,\rho_0)$. We can assume that $r_{0}<1<|\lambda_{I}|$. If not, we dimish $r_0$ as desired. We put
$$ \phi_{I,j}(\epsilon,t):=\mathcal{L}_{q;1}^{\lambda_I}(\tau\mapsto W_{U_I,V_I,j}(\epsilon,\tau))(\epsilon t).  $$
Similar arguments as the ones used in the proof of Theorem~\ref{teo572} and taking into account Assumption(B), Proposition~\ref{prop11} and estimates in (\ref{e658}), we deduce that $\phi_{I,j}(\epsilon,t)$ is holomorphic and bounded on $U_Iq^{-\N}\times\mathcal{T}$ for every $I\in\mathcal{I},0\le j\le S-1$.

The following assumption is related to technical reasons appearing in the proof of the following theorem.
 
\textbf{Assumption (C)}:
There exist $a_1,a_2,b_1,b_2,d_1,d_2>0$, $0<\xi,\overline{\xi}<1$ such that
\begin{enumerate}
\item[(C.1)]\hfill $\log|q|<\frac{b_1}{b_2}$,\phantom{mmmmmmmmmmmmmmmmm}
\item[(C.2)]\hfill $\log|q|+\frac{\xi b_1}{2b_2}+\frac{d_1}{d_2}\Big(M-\frac{\xi}{2\log|q|}\Big)>0$,\phantom{mmmmmmmmmmmmmmmmm}
\item[(C.3)]\hfill $M-\frac{\xi}{2\log|q|}+\frac{d_2}{d_1}\log|q|<0$,\phantom{mmmmmmmmmmmmmmmmm}
\item[(C.4)]$$A_1\Big(1-\frac{d_2\log|q|}{d_1(\frac{\xi}{2\log|q|}-M)}\Big)>\frac{C^2}{4\overline{\xi}\log|q|(\frac{\xi}{2\log|q|}-M)}+\frac{Ca_{2}}{a_{1}}.$$
\end{enumerate}

Next remark clarifies availability of these constants for a posed problem.

\textbf{Remark:} Assumptions (A), (B) and (C) hardly depend on the choice of $q$ whose modulus must rest near 1. This assumptions on the constants are verified when taking $\log|q|=1/16$, $M=1$, $A_1=2$, $C=1$, $\xi=1/2$, $\overline{\xi}=1/2$, $a_1=9$, $a_2=1$, $b_1=b_2=1$, $d_1=1$, $d_2=12$, 
$\overline{\xi}=\xi=1/2$, $C=1/4$, $a_{1}=a_{2}=1$ and $M=1/(\log|q|8)$. Then, the previous theorem provides a solution for the equation
$$\epsilon t\partial^4_{z}X_{I}(\epsilon,qt,z)+\partial_{z}^4 X_{I}(\epsilon,t,z)=\big(b_{00}(\epsilon)+b_{01}(\epsilon)z\big)t^3X_{I}(\epsilon,q^3t,zq^{-17})+b_{10}z\partial_{z}X_{I}(\epsilon,q^4t,zq^{-21}),$$
with $b_{00},b_{01},b_{11}$ being holomorphic functions near the origin.

\begin{theo}\label{teo572}
Let Assumption (A) be fulfilled by the integers $m_{0,k},m_{1,k}$, for $0\le k\le S-1$ and also assumptions (B) and (C) for $M,A_{1},C$. We consider the problem (\ref{ee1})+(\ref{ii1})  where the initial conditions are constructed as above. Then, for every $I\in\mathcal{I}$, the problem (\ref{ee1})+(\ref{ii1}) has a solution $X_{I}(\epsilon,t,z)$ which is holomorphic and bounded in $U_{I}q^{-\N}\times\mathcal{T}\times \C$. 

Moreover, for every $\rho>0$, if $I,I'\in\mathcal{I}$ are such that $U_{I}q^{-\N}\cap U_{I'}q^{-\N}\neq\emptyset$ then there exists a positive constant $C_{1}=C_{1}(\rho)>0$ such that 
$$|X_{I}(\epsilon, t, z )-X_{I'}(\epsilon, t, z)|\le C_{1}e^{-\frac{1}{A}\log^2|\epsilon|},\quad (\epsilon, t,z)\in (U_{I}q^{-\N}\cap U_{I'}q^{-\N})\times \mathcal{T}\times D(0,\rho),$$
with $\frac{1}{A}=(1-\overline{\xi})(\frac{\xi}{2\log|q|}-M)$ with $\xi,\overline{\xi}$ chosen as in Assumption (C).

\end{theo}
\begin{proof}
Let $\delta>0$ and $I\in\mathcal{I}$. We consider the Cauchy problem (\ref{fe2}) with initial conditions $(\partial_{z}^{j}W)(\epsilon,\tau,0)=W_{j}(\epsilon,\tau)$ for $0\le j\le S-1$. From Theorem~\ref{teorema2} we obtain the existence of a unique formal solution $W(\epsilon,\tau,z)=\sum_{\b\ge0}W_{\b}(\epsilon,\tau)\frac{z^{\b}}{\b}\in\mathcal{O}(D(0,r_{0})\setminus\{0\}\times \dot{D}_{\rho_0})[[z]]$ and positive constants $C_3>0$ and $0<\delta_1<1$ such that 
\begin{equation}\label{aa}
|W_{\b}(\epsilon,\tau)|\le C_3\b!\Big(\frac{|q|^{2A_1S}}{\delta_1}\Big)^{\b}|\epsilon|^{-C\b}e^{M\log^2\left|\frac{\tau}{\epsilon}\right|}|q|^{-A_{1}\b^2},\quad\b\ge0,
\end{equation}
for $(\epsilon,\tau)\in(D(0,r_0)\setminus\{0\})\times \dot{D}_{\rho_0}$.

Moreover, from Theorem~\ref{teorema1} we get that the coefficients $W_{\b}(\epsilon,\tau)$ can be extended to holomorphic functions defined in $U_{I}q^{-\N}\times V_{I}q^{\R_{+}}$ and also the existence of positive constants $C_2$ and $0<\delta_2<1$ such that
\begin{equation}\label{bb}
|W_{\b}(\epsilon,\tau)|\le C_2\b!\Big(\frac{|q|^{2A_1S}}{\delta_2}\Big)^{\b}\left|\frac{\tau}{\epsilon}\right|^{C\b}e^{M\log^2\left|\frac{\tau}{\epsilon}\right|}|q|^{-A_{1}\b^2},\quad\b\ge0,
\end{equation}
for $(\epsilon,\tau)\in U_{I}q^{-\N}\times V_{I}q^{\R_{+}}$.

We choose $\lambda_{I}\in V_{I}\cap D(0,\rho_{0})$. In the following estimates we will make use of the fact that $|\epsilon|\le|\lambda_{I}|$ for every $\epsilon\in D(0,r_{0}\setminus\{0\})$. Proposition~\ref{prop11} allows us to calculate the $q-$Laplace transform of $W_{\b}$ with respect to $\tau$ for every $\b\ge 0$, $\mathcal{L}_{q;1}^{\lambda_{I}}(W_{\b})(\epsilon,\tau)$. It defines a holomorphic function in $U_{I}q^{-\N}\times \mathcal{R}_{\lambda_{I},q,\delta}$. From the fact that $\{(V_{I})_{I\in\mathcal{I}},\mathcal{T}\}$ is chosen to be a family associated to the good covering $(U_{I}q^{-\N})_{I\in\mathcal{I}}$ we derive that the function 
$$(\epsilon,t)\mapsto \mathcal{L}_{q;1}^{\lambda_{I}}(W_{\b})(\epsilon,\epsilon t)$$
is a holomorphic and bounded function defined in $U_{I}q^{-\N}\times \mathcal{T}$. 
We can define, at least formally, 
\begin{equation}\label{e710}
X_{I}(\epsilon,t,z):= \sum_{\b\ge 0}\mathcal{L}_{q;1}^{\lambda_{I}}(W_{\b})(\epsilon,\epsilon t)\frac{z^{\b}}{\b!},
\end{equation}
in $\mathcal{O}(U_{I}q^{-\N}\times \mathcal{T})[[z]]$.
If $X_{I}(\epsilon,t,z)$ were a holomorphic function in $U_{I}q^{-\N}\times\mathcal{T}\times\C$, then Proposition~\ref{proposicion2} would allow us to affirm that (\ref{e710}) is an actual solution of (\ref{ee1})+(\ref{ii1}).
In order to end the first part of the proof it rests to demonstrate that (\ref{e710}) defines in fact a bounded holomorphic function in $U_{I}q^{-\N}\times \mathcal{T}\times \C$. 
Let $(\epsilon,t)\in U_{I}q^{-\N}\times \mathcal{T}$ and $\b\ge0$. We have

$$|\mathcal{L}_{q;1}^{\lambda_{I}}W_{\b}(\epsilon,\epsilon t)|\le |\mathcal{L}_{q;1,+}^{\lambda_{I}}W_{\b}(\epsilon,\epsilon t)|+|\mathcal{L}_{q;1,-}^{\lambda_{I}}W_{\b}(\epsilon,\epsilon t)|,$$
where
$$\mathcal{L}_{q;1,+}^{\lambda_{I}}W_{\b}(\epsilon,\epsilon t)=\frac{\log(q)}{\pi_{q}}\int_{0}^{\infty}\frac{W_{\b}(\epsilon,q^s\lambda_{I})}{\Theta(\frac{q^{s}\lambda_{I}}{\epsilon t})}ds,\quad \mathcal{L}_{q;1,-}^{\lambda_{I}}W_{\b}(\epsilon,\epsilon t)=\frac{\log(q)}{\pi_{q}}\int_{-\infty}^{0}\frac{W_{\b}(\epsilon,q^s\lambda_{I})}{\Theta(\frac{q^{s}\lambda_{I}}{\epsilon t})}ds.$$
We now establish bounds for both branches of $q-$Laplace transform.

$$|\mathcal{L}_{q;1,+}^{\lambda_{I}}W_{\b}(\epsilon,\epsilon t) |\le\frac{|\log q|}{|\pi_{q}|}\int_{0}^{\infty}\left|\frac{W_{\b}(\epsilon,q^{s}\lambda_{I})}{\Theta(\frac{q^s\lambda_{I}}{\epsilon t})}\right|ds.$$
Let $0<\xi<1$ as in Assumption (C). From (\ref{bb}) and (\ref{e376}), the previous integral is bounded by
\begin{align*}
&\le \frac{|\log q|}{|\pi_{q}|}\int_{0}^{\infty}\frac{C_2\b! \left(\frac{|q|^{2A_{1}S}}{\delta_{2}}\right)^{\b}\left|\frac{q^s\lambda_{I}}{\epsilon}\right|^{C\b}e^{M\log^{2}\left|\frac{q^s\lambda_{I}}{\epsilon}\right|}|q|^{-A_{1}\b^2}}{C_{\xi}\exp(\frac{\xi\log^{2}|\frac{q^s\lambda}{\epsilon t}|}{2\log|q|})}ds\\
&\le\frac{|\log q|}{|\pi_{q}|} \frac{C_2}{C_{\xi}}\b!\left(\frac{|q|^{2A_{1}S}}{\delta_{2}}\right)^{\b}\left|\frac{\lambda_{I}}{\epsilon}\right|^{C\b}|q|^{-A_{1}\b^2}\int_{0}^{\infty}\frac{ |q|^{Cs\b}e^{M\log^{2}\left|\frac{q^s\lambda_{I}}{\epsilon}\right|}}{\exp({\frac{\xi\log^{2}|\frac{q^s\lambda_{I}}{\epsilon t}|}{2\log|q|}})}ds.
\end{align*}

Let $a_{1},a_{2}$ as in Assumption (C.4).

From $(a_1s-a_2\b)^2\ge0$ and 4. in Definition~\ref{definicion536}, the previous inequality is upper bounded by
\begin{equation}\label{e512}
\mathcal{A}\int_{0}^{\infty}|q|^{-Bs^{2}}e^{(M-\frac{\xi}{2\log|q|})\log^{2}|\lambda_{I}/\epsilon|}e^{((2M\log|q|-\xi)\log|\lambda_{I}/\epsilon|+\xi\log|t|)s}ds,
\end{equation}
where $0<B=\xi\log|q|/2-M\log^2|q|-\frac{Ca_{1}}{2a_{2}}$ and
$$\mathcal{A}=\frac{|\log q|}{|\pi_{q}|}\frac{C_2}{C_{\xi}}\b!\left(\frac{|q|^{2A_{1}S}}{\delta_{2}}\right)^{\b}\left|\frac{\lambda_{I}}{\epsilon}\right|^{C\b}|q|^{-A_{1}\b^2+\frac{Ca_{2}\b^2}{2a_{1}}}e^{-\frac{\xi\log^2|t|}{2\log|q|}}e^{\frac{\xi\log|\lambda_{I}/\epsilon|\log|t|}{\log|q|}}.$$

The previous integral is uniformly bounded for $\epsilon\in D(0,r_{0})\setminus\{0\}$ and $t\in\mathcal{T}$ from hypotheses made on these sets. The expression in (\ref{e512}) can be bounded by 
$$ \frac{|\log q|}{|\pi_{q}|}\frac{C'_2}{C_{\xi}}\b!\left(\frac{|q|^{2A_{1}S}}{\delta_{2}}\right)^{\b}\left|\frac{\lambda_{I}}{\epsilon}\right|^{C\b}e^{(M-\frac{\xi}{2\log|q|})\log^2|\lambda_{I}/\epsilon|}|q|^{-A_{1}\b^2+\frac{Ca_{2}\b^2}{2a_{1}}}e^{-\frac{\xi\log^2|t|}{2\log|q|}}e^{\frac{\xi\log|\lambda_{I}/\epsilon|\log|t|}{\log|q|}},$$
for an appropriate constant $C'_{2}>0$.

The function $s\mapsto s^{\gamma\b}e^{-\alpha\log^{2}(s)}$ takes its maximum at $s=e^{\gamma\b/(2\alpha)}$ so each element in the image set is bounded by $e^{(\gamma\b)^{2}/(4\alpha)}$.
Taking this to the expression above we get 
$$|\mathcal{L}_{q;1,+}^{\lambda_{I}}W_{\b}(\epsilon ,\epsilon t)|\le \frac{|\log q|}{|\pi_{q}|}\frac{C''_{2}}{C_{\xi}}\b!\left(\frac{|q|^{2A_{1}S}}{\delta_{2}}\right)^{\b}|q|^{-A_{1}\b^2+\frac{Ca_{2}\b^2}{2a_{1}}+\frac{C^2\b^2}{4\log|q|(\xi/(2\log|q|)-M)}},$$
for certain $C''_2>0$.

Assumption (C.4) applied to the last term in the previous expression allows us to deduce that the sum
\begin{equation}\label{e619}
 \sum_{\b\ge 0}|\mathcal{L}_{q;1,+}^{\lambda_{I}}W_{\b}(\epsilon ,\epsilon t)|\frac{|z|^{\b}}{\b!}
\end{equation}
converges in the variable $z$ uniformly in the compact sets of $\C$. 

We now study $\mathcal{L}_{q;1,-}^{\lambda_{I}}W_{\b}(\epsilon, \epsilon t)$.
We have

$$|\mathcal{L}_{q;1,-}^{\lambda_{I}}W_{\b}(\epsilon ,\epsilon t)|\le\frac{|\log q|}{|\pi_{q}|}\int_{-\infty}^{0}\left|\frac{W_{\b}(\epsilon,q^{s}\lambda_{I})}{\Theta(\frac{q^s\lambda_{I}}{\epsilon t})}\right|ds.$$
From (\ref{e325}) and (\ref{e376}) the previous integral is bounded by 
$$\frac{|\log q|}{|\pi_{q}|}\phantom{x}\int_{-\infty}^{0}\frac{
 C_{3}\b! \left(\frac{|q|^{2A_{1}S}}{\delta_{1}}\right)^{\b}|\epsilon|^{-C\b}e^{M\log^{2}\left|\frac{q^s\lambda_{I}}{\epsilon}\right|}|q|^{-A_{1}\b^2}}{C_{\xi}e^{\frac{\xi\log^2|\frac{q^s\lambda_{I}}{\epsilon t}|}{2\log|q|}}}ds.$$

Similar calculations as in the first part of the proof resting on Assumption (C) can be followed so that the series
\begin{equation}\label{e657}
\sum_{\b\ge0}\mathcal{L}_{q;1,-}^{\lambda_{I}}W_{\b}(\epsilon ,\epsilon t)\frac{z^\b}{\b!}
\end{equation}
is uniformly convergent with respect to the variable $z$ in the compact sets of $\C$, for $(\epsilon,t)\in U_{I}q^{-\N}\times\mathcal{T}$. We will not enter into details not to repeat calculations.

The estimates (\ref{e619}) and (\ref{e657}) attain convergence of the series in (\ref{e710}) for every $z\in\C$. Boundness of the $q-$Laplace transform with respect to $\epsilon$ is guaranteed so the first part of the result is achieved.

Let $I,I'\in\mathcal{I}$ such that $U_{I}q^{-\N}\cap U_{I'}q^{-\N}\neq\emptyset$ and $\rho>0$. For every $(\epsilon, t,z)\in (U_{I}q^{-\N}\cap U_{I'}q^{-\N})\times\mathcal{T}\times D(0,\rho)$ we have
\begin{equation}\label{e682}
|X_{I}(\epsilon, t, z)- X_{I'}(\epsilon, t,z)|\le \sum_{\b\ge0}|\mathcal{L}_{q;1}^{\lambda_{I}}W_{\b}(\epsilon, \epsilon t)-\mathcal{L}_{q;1}^{\lambda_{I'}}W_{\b}(\epsilon, \epsilon t)|\frac{\rho^{\b}}{\b!}.
\end{equation}

We can write
\begin{equation}\label{aux5}
\mathcal{L}_{q;1}^{\lambda_{I}}W_{\b}(\epsilon, \epsilon t)-\mathcal{L}_{q;1}^{\lambda_{I'}}W_{\b}(\epsilon, \epsilon t)=\frac{\log(q)}{\pi_{q}}\Big(\int_{\gamma_{1}}\frac{W_{\b}(\epsilon,\xi)}{\Theta(\xi/\epsilon t)}\frac{d\xi}{\xi}-\int_{\gamma_{2}}\frac{W_{\b}(\epsilon,\xi)}{\Theta(\xi/\epsilon t)}\frac{d\xi}{\xi}+\int_{\gamma_{3}-\gamma_{4}}\frac{W_{\b}(\epsilon,\xi)}{\Theta(\xi/\epsilon t)}\frac{d\xi}{\xi}   \Big)
\end{equation}
where the path $\gamma_{1}$ is given by $s\in(0,\infty)\mapsto q^s\lambda_{I}$, $\gamma_{2}$ is given by $s\in(0,\infty)\mapsto q^s\lambda_{I'}$, $\gamma_{3}$ is $s\in(-\infty,0)\mapsto q^s\lambda_{I}$ and $\gamma_{4}$ is $s\in(-\infty,0)\mapsto q^s\lambda_{I'}$.

Without loss of generality, we can assume that $|\lambda_{I}|=|\lambda_{I'}|$.

For the first integral we deduce
$$\Big|\int_{\gamma_{1}}\frac{W_{\b}(\epsilon,\xi)}{\Theta(\xi/\epsilon t)}\frac{d\xi}{\xi}\Big|\le|\log(q)|\int_{0}^{\infty}\frac{|W_{\b}(\epsilon,q^{s}\lambda_{I})|}{|\Theta(\frac{q^{s}\lambda_{I}}{\epsilon t})|}ds.$$
Similar estimates as in the first part of the proof lead us to bound the right part of previous inequality by
$$\frac{C'''_{2}}{C_{\xi}}\b!\Big(\frac{|q|^{2A_{1}S}}{\delta_{2}}\Big)^{\b}\Big|\frac{\lambda_{I}}{\epsilon}\Big|^{C\b}|q|^{-A_{1}\b^2+\frac{Ca_{2}}{2a_{1}}\b^2}e^{(M-\frac{\xi}{2\log|q|})\log^2|\lambda_{I}/\epsilon|},$$
for certain $C'''_2>0$.
For any $\overline{\xi}\in(0,1)$ we have
$$\Big|\frac{\lambda_{I}}{\epsilon}\Big|^{C\beta}e^{\overline{\xi}(M-\frac{\xi}{2\log|q|})\log^{2}|\lambda_{I}/\epsilon|}\le e^{\frac{C^{2}\b^{2}}{4\overline{\xi}(\frac{\xi}{2\log|q|}-M)}},\quad \b\ge0.$$
This yields 
\begin{equation}\label{e699}
\int_{\gamma_{1}}\Big|\frac{W_{\b}(\epsilon,q^{s}\lambda_{I})}{\Theta(\frac{q^{s}\lambda_{I}}{\epsilon t})}\Big|ds\le\frac{C'''_{2}}{C_{\xi}}\b! \Big(\frac{|q|^{2A_{1}S}}{\delta_2}\Big)^{\b}|q|^{(-A_{1}+\frac{Ca_{2}}{2a_{1}}+\frac{C^{2}}{4\overline{\xi}(\frac{\xi}{2\log|q|}-M)})\b^2}e^{(1-\overline{\xi})(M-\frac{\xi}{2\log|q|})\log^{2}|\lambda_{I}/\epsilon|}.
\end{equation}
We choose $\overline{\xi}$ as in Assumption (C).

The integral corresponding to the path $\gamma_{2}$ can be bounded following identical steps.

We now give estimates concerning $\gamma_{3}-\gamma_{4}$. It is worth saying that the function in the integrand is well defined for $(\epsilon,\tau)\in D(0,r_{0})\setminus\{0\}\times \dot{D}_{\rho_{0}}$ and does not depend on the index $I\in\mathcal{I}$. This fact and Cauchy Theorem allow us to write for any $n\in\N$
$$\int_{\Gamma_{n}}\frac{W_{\b}(\epsilon,\xi)}{\Theta(\xi/\epsilon t)}\frac{d\xi}{\xi}=0,$$
where $\Gamma_{n}=\gamma_{n,1}+\gamma_{5}-\gamma_{n,2}-\gamma_{n,3}$ is the closed path defined in the following way: $s\in[-n,0]\mapsto\gamma_{n,1}(s)=\lambda_{I}q^{s}$,$\gamma_{5}$ is the arc of circunference from $\lambda_{I}$ to $\lambda_{I'}$, $s\in[-n,0]\mapsto\gamma_{n,2}(s)=\lambda_{I'}q^{s}$ and $\gamma_{n,3}$ is the arc of circunference from $\lambda_{I}q^{-n}$ to $\lambda_{I'}q^{-n}$.
Taking $n\to\infty$ we derive
\begin{equation}\label{aux1}
0=\lim_{n\to\infty}\int_{\Gamma_{n}}\frac{W_{\b}(\epsilon,\xi)}{\Theta(\xi/\epsilon t)}\frac{d\xi}{\xi}=\lim_{n\to\infty}\int_{\gamma_{n,1}+\gamma_{5}-\gamma_{n,2}}\frac{W_{\b}(\epsilon,\xi)}{\Theta(\xi/\epsilon t)}\frac{d\xi}{\xi}-\lim_{n\to\infty}\int_{\gamma_{n,3}}\frac{W_{\b}(\epsilon,\xi)}{\Theta(\xi/\epsilon t)}\frac{d\xi}{\xi}.
\end{equation}
Usual estimates lead us to prove that
\begin{equation}\label{aux2}
\lim_{n\to\infty}\int_{\gamma_{n,3}}\frac{W_{\b}(\epsilon,\xi)}{\Theta(\xi/\epsilon t)}\frac{d\xi}{\xi}=0.
\end{equation}
Moreover,
\begin{equation}\label{aux3}
\lim_{n\to\infty}\int_{\gamma_{n,1}+\gamma_{5}-\gamma_{n,2}}\frac{W_{\b}(\epsilon,\xi)}{\Theta(\xi/\epsilon t)}\frac{d\xi}{\xi}=\int_{\gamma_{3}+\gamma_{5}-\gamma_{4}}\frac{W_{\b}(\epsilon,\xi)}{\Theta(\xi/\epsilon t)}\frac{d\xi}{\xi}.
\end{equation}
From (\ref{aux1}), (\ref{aux2}) and (\ref{aux3}) we obtain
$$\int_{\gamma_{3}-\gamma_{4}}\frac{W_{\b}(\epsilon,\xi)}{\Theta(\xi/\epsilon t)}\frac{d\xi}{\xi}=\int_{-\gamma_{5}}\frac{W_{\b}(\epsilon,\xi)}{\Theta(\xi/\epsilon t)}\frac{d\xi}{\xi}.$$

Taking into account Definition~\ref{definicion536} and (\ref{aa}) we derive

$$\Big|\int_{-\gamma_{5}}\frac{W_{\b}(\epsilon,\xi)}{\Theta(\xi/\epsilon t)}\frac{d\xi}{\xi}\Big|=\Big|\int_{\theta_{I'}}^{\theta_{I}}\frac{W_{\b}(\epsilon,|\lambda_{I}|e^{i\theta})}{\Theta(\frac{|\lambda_{I}|e^{i\theta}}{\epsilon t})}d\theta\Big|,$$
where $\theta_{I}=\hbox{arg}(\lambda_{I})$, $\theta_{I'}=\hbox{arg}(\lambda_{I'})$. This last expression is bounded by
\begin{align*}
&\frac{\hbox{length}(\gamma_{5})C_{3}}{C_{\xi}}\b!\Big(\frac{|q|^{2A_{1}S}}{\delta_{1}}\Big)^{\b}|\epsilon|^{-C\b}\frac{e^{M\log^2\left|\frac{\lambda_{I}}{\epsilon}\right|}}{e^{\frac{\xi}{2\log|q|}\log^2\left|\frac{\lambda_{I}}{\epsilon t}\right|}}|q|^{-A_{1}\b^2}\\
&\le C'_{3}\b!\Big(\frac{|q|^{2A_{1}S}}{\delta_{1}}\Big)^{\b}|\epsilon|^{-C\b}e^{(M-\frac{\xi}{2\log|q|})\log^2\left|\frac{\lambda_{I}}{\epsilon}\right|}|q|^{-A_{1}\b^2}\\
&\le C'_{3}\b!\Big(\frac{|q|^{2A_{1}S}}{\delta_{1}}\Big)^{\b}|\epsilon|^{-C\b}e^{\overline{\xi}(M-\frac{\xi}{2\log|q|})\log^2|\epsilon|}|q|^{-A_{1}\b^2}e^{(1-\overline{\xi})(M-\frac{\xi}{2\log|q|})\log^2|\epsilon|}.
\end{align*}
for adequate positive constants $C_{3},C'_{3}$. From standard estimates we achieve
\begin{equation}\label{e708}
\Big|\int_{\gamma_{3}-\gamma_{4}}\frac{W_{\b}(\epsilon,\xi)}{\Theta(\xi/\epsilon t)}\frac{d\xi}{\xi}\Big|\le C'_3 \b!\Big(\frac{|q|^{2A_{1}S}}{\delta_{1}}\Big)^{\b}|q|^{-A_{1}\b^2}e^{\frac{C^2}{4\overline{\xi}(\frac{\xi}{2\log|q|}-M)}\b^2}e^{(1-\overline{\xi})(M-\frac{\xi}{2\log|q|})\log^2|\epsilon|}.
\end{equation}

From (\ref{e682}), (\ref{aux5}), (\ref{e699}), (\ref{e708}) and Assumption (C.4) we conclude the existence of a positive constant $C'_{1}>0$ such that
$$|X_{I}(\epsilon,t,z)-X_{I'}(\epsilon,t,z)|\le C'_{1}\sum_{\b\ge0}\b!\Big(\frac{|q|^{2A_{1}S}}{\delta_{0}}\Big)^{\b}|q|^{\Big(-A_{1}+\frac{Ca_{1}}{2a_{2}}+\frac{C^2}{4\overline{\xi}\log|q|(\frac{\xi}{2\log|q|}-M)}\Big)\b^2}\times$$
$$\times e^{(1-\overline{\xi})(M-\frac{\xi}{2\log|q|})\log^2|\epsilon|}\frac{\rho^{\b}}{\b!}
\le C_{1}e^{(1-\overline{\xi})(M-\frac{\xi}{2\log|q|})\log^2|\epsilon|},$$
for every $(\epsilon, t, z)\in (U_{I}q^{-\N}\cap U_{I'}q^{-\N})\times\mathcal{T}\times D(0,\rho)$, with $\delta_{0}=\min\{\delta_{1},\delta_{2}\}$.

\end{proof}

\end{subsection}
\end{section}

\begin{section}{A $q-$Gevrey Malgrange-Sibuya type theorem}
In this section we obtain a $q-$Gevrey version of the so called Malgrange-Sibuya theorem which allows us to reach our final main achievement: the existence of a formal series solution of problem (\ref{ee1})+(\ref{ii1}) which asymptotically represents the actual solutions obtained in Theorem~\ref{teo572}, meaning that for every $I\in\mathcal{I}$, $X_{I}$ admits this formal solution as its $q-$Gevrey asymptotic expansion in the variable $\epsilon$.

In~\cite{malek}, a Malgrange-Sibuya type theorem appears with similar aims as in this work. We complete the information there giving bounds on the estimates appearing for the $q-$asymptotic expansion. This mentioned work heavily rests on the theory developed by J-P. Ramis, J. Sauloy and C. Zhang in~\cite{ramissauloyzhang}.

In the present work, although $q-$Gevrey bounds are achieved, the $q-$Gevrey type involved will not be preserved suffering increasement on the way.

The nature of the proof relies in the one concerning classical Malgrange-Sibuya theorem for Gevrey asymptotics which can be found in~\cite{malgrange3}. 

Let $\mathbb{H}$ be a complex Banach space.
\begin{defin}\label{defi679}
Let $U$ be a bounded open set in $\C^{\star}$ and $A>0$. We say a holomorphic function $f:Uq^{-\N}\to \mathbb{H}$ admits $\hat{f}=\sum_{n\ge 0}f_{n}\epsilon^{n}\in\mathbb{H}[[\epsilon]]$ as its $q-$Gevrey asymptotic expansion of type $A$ in $Uq^{-\N}$ if for every compact set $K\subseteq U$ there exist $C_{1},H>0$ such that
$$\left\|f(\epsilon)-\sum_{n=0}^{N}f_{n}\epsilon^n\right\|_{\mathbb{H}}\le C_{1}H^{N}|q|^{A\frac{N^2}{2}}\frac{|\epsilon|^{N+1}}{(N+1)!},\quad N\ge0, $$
for every $\epsilon\in Kq^{-\N}$.
\end{defin}
 
The following proposition can be found, under slightly modifications in Section 4 of~\cite{ramissauloyzhang}.
\begin{prop}\label{prop3}
Let $A>0$ and $U\subseteq\C^{\star}$ be an open and bounded set. Let $f:Uq^{-\N}\to\mathbb{H}$ be a holomorphic function that admits a formal power series $\hat{f}\in\mathbb{H}[[\epsilon]]$ as its $q-$Gevrey asymptotic expansion of type $A$ in $Uq^{-\N}$. Then, if $\hat{f}^{(k)}$ denotes de $k-$th formal derivative of $\hat{f}$ for every $k\in\N$, we have that $f^{(k)}$ admits $\hat{f}^{(k)}$ as its $q-$Gevrey asymptotic expansion of type $A$ in $Uq^{-\N}$.
\end{prop}

\begin{prop}\label{prop704}
Let $A>0$ and $f:Uq^{-\N}\to\mathbb{H}$ a holomorphic function in $Uq^{-\N}$. Then,
\begin{enumerate}
\item[i)] If $f$ admits $\hat{0}$ as its $q-$Gevrey expansion of type $A$, then for every compact set $K\subseteq U$ there exists $C_{1}>0$ with
$$\left\|f(\epsilon)\right\|_{\mathbb{H}}\le C_{1}e^{-\frac{1}{\tilde{a}}\frac{1}{2\log|q|}\log^{2}|\epsilon|} ,$$
for every $\epsilon\in Kq^{-\N}$ and every $\tilde{a}>A$.
\item[ii)] If for every compact set $K\subseteq U$ there exists $C_{1}>0$ with
$$\left\|f(\epsilon)\right\|_{\mathbb{H}}\le C_{1}e^{-\frac{1}{A}\frac{1}{2\log|q|}\log^{2}|\epsilon|} ,$$
for every $\epsilon\in Kq^{-\N}$ then $f$ admits $\hat{0}$ as its $q-$Gevrey asymptotic expansion of type $\tilde{a}$ in $Uq^{-\N}$, for every $\tilde{a}>A$.
\end{enumerate}
\end{prop}
\begin{proof}
Let $C_{1},H,A>0$ and $\epsilon\in\C^{\star}$. The function $$G(x)=C_{1}\exp(\log(H)x+\frac{\log|q|A}{2}x^{2}+(x+1)\log|\epsilon|)$$ is minimum for $x>0$ at $x_{0}=\frac{-\log(H)-\log|\epsilon|}{A\log|q|}$. We deduce both results from standard calculations.
\end{proof}

\begin{defin}\label{def745}
Let $(U_{I})_{I\in\mathcal{I}}$ be a good covering at 0 (see Definition~\ref{def527}), and $g_{I,I'}:U_{I}q^{-\N}\cap U_{I'}q^{-\N}\to \mathbb{H}$ a holomorphic function in $U_{I}q^{-\N}\cap U_{I'}q^{-\N}$ for $I,I'\in\mathcal{I}$ when the intersection is not empty. The family $(g_{I,I'})_{(I,I')\in\mathcal{I}^2}$ is a $q-$Gevrey $\mathbb{H}-$cocycle of type $A>0$ attached to a good covering $(U_Iq^{-\N})_{I\in\mathcal{I}}$ if the following properties are satisfied:
\begin{enumerate}
\item[1.] $g_{I,I'}$ admits $\hat{0}$ as its $q-$Gevrey asymptotic expansion of type $A>0$ on $U_{I}q^{-\N}\cap U_{I'}q^{-\N}$ for every $(I,I')\in\mathcal{I}$.
\item[2.] $g_{I,I'}(\epsilon)=-g_{I',I}(\epsilon)$ for every $(I,I')\in\mathcal{I}$, and $\epsilon\in U_{I}q^{-\N}\cap U_{I'}q^{-\N}$.
\item[3.] We have $g_{I,I''}(\epsilon)=g_{I,I'}(\epsilon)+g_{I',I''}(\epsilon)$ for all $\epsilon\in U_{I}q^{-\N}\cap U_{I'}q^{-\N} \cap U_{I''}q^{-\N}$, $I,I',I''\in\mathcal{I}$.
\end{enumerate}
\end{defin}

Let $\rho>0$ and $\mathcal{T}\subseteq\C^{\star}$ be an open and bounded set. $\mathbb{H}_{\mathcal{T},\rho}$ stands for the Banach space of holomorphic and bounded functions in $\mathcal{T}\times D(0,\rho)$ with the supremum norm.

\begin{prop}\label{prop754}
Let $\rho>0$. We consider the family $(X_{I}(\epsilon,t,z))_{I\in\mathcal{I}}$ constructed in Theorem~\ref{teo572}. Then, the set of functions $(g_{I,I'}(\epsilon))_{(I,I')\in\mathcal{I}^2}$ defined by 
$$g_{I,I'}(\epsilon):=(t,z)\in\mathcal{T}\times D(0,\rho)\mapsto X_{I'}(\epsilon,t,z)-X_{I}(\epsilon, t,z)$$
for $I,I'\in\mathcal{I}$ is a $q-$Gevrey $\mathbb{H}_{\mathcal{T},\rho}$-cocycle of type $\tilde{A}$ for every $\tilde{A}>A:=\frac{1}{(1-\overline{\xi})(\frac{\xi}{2\log|q|}-M)}$ attached to the good covering $(U_Iq^{-\N})_{I\in\mathcal{I}}$.
\end{prop}
\begin{proof}
The first property in Definition~\ref{def745} directly comes from Theorem~\ref{teo572} and Proposition~\ref{prop704}. The other two are verified by construction of the cocycle.
\end{proof}
We recall several definitions and an extension result from~\cite{bonetbraunmeisetaylor} which will be crucial in our work.
\begin{defin}A continuous increasing function $w:[0,\infty)\to[0,\infty)$ is a weight function if it satisfies
\begin{enumerate}
\item[$(\alpha)$] there exists $k\ge1$ with $w(2t)\le k(w(t)+1)$ for all $t\ge0$,
\item[($\beta$)] $\int_{0}^{\infty}\frac{w(t)}{1+t^2}dt<\infty,$
\item[($\gamma$)] $\lim_{t\to\infty}\frac{\log t}{w(t)}=0,$
\item[($\delta$)] $\phi:t\mapsto w(e^t)$ is convex.
\end{enumerate}
The Young conjugate associated to $\phi$, $\phi^\star:[0,\infty)\to\R$ of $\phi$ is defined by
$$\phi^{\star}(y):=\sup\{xy-\phi(x):x\ge0\}.$$
\end{defin}
\begin{defin}\label{def773}
Let $K$ be a nonempty compact set in $\R^{2}$. A jet on $K$ is a family $F=(f^{\alpha})_{\a\in\N^{2}}$ where $f^{\a}:K\to\C$ is a continuous function on $K$ for each $\a\in\N^{2}$.

Let $w$ be a weight function. A jet $F=(f^{\a})_{\a\in\N^{2}}$ on $K$ is said to be a $w-$Whitney jet (of Roumieu type) on $K$ if there exist $m>0$ and $M>0$ such that
$$\left\|f\right\|_{K,1/m}:=\sup_{x\in K,\a\in\N^{2}}|f^{\a}(x)|\exp(-\frac{1}{m}\phi^{\star}(m|\a|))\le M,$$
and for every $l\in\N$, $\a\in\N^{2}$ with $|\alpha|\le l$ and $x,y\in K$ one has
$$|(R^{l}_{x}F)_{\alpha}(y)|\le M\frac{|x-y|^{l+1-|\alpha|}}{(l+1-|\a|)!}\exp(\frac{1}{m}\phi^{\star}(m(l+1))), $$
where $(R_{x}^{l}F)_{\a}(y):=f^{\a}(y)-\sum_{|\a+\b|\le l}\frac{1}{\b!}f^{\a+\b}(x)(y-x)^{\b}$.
\end{defin}

$\mathcal{E}_{\{w\}}(K)$ denotes the linear space of $w-$Whitney jets on $K$.

\begin{defin}
Let $K\subseteq\R^2$ be a nonempty compact set and $w$ a weight function in $K$. A continuous function $f:K\to\C$ is $w-\mathcal{C}^{\infty}$ in the sense of Whitney in $K$ if there exists a $w-$Whitney jet on $K$, $(f^{\alpha})_{\alpha\in\N^2}$ such that $f^{(0,0)}=f$.

For an open set $\Omega\in\R^{2}$ we define
$$\mathcal{E}_{\{w\}}(\Omega):=\{f\in\mathcal{C}^{\infty}(\Omega): \forall K \subseteq \Omega, K\hbox{ compact },\exists m>0, \left\|f\right\|_{K,1/m}<\infty\}.$$
\end{defin}

The following result establishes conditions on a weight function so that a jet in $\mathcal{E}_{\{w\}}(K)$ can be extended to an element in $\mathcal{E}_{\{w\}}(\R^2)$ if and only if $w$ is a strong weight function.
\begin{theo}[Corollary 3.10,~\cite{bonetbraunmeisetaylor}]\label{teo792}
For a given weight function $w$, the following statements are equivalent:
\begin{enumerate}
\item For every nonempty closed set $K$ in $\R^{2}$ the restriction map sending a function $f\in\mathcal{E}_{\{w\}}(\R^{2})$ to the family of derivatives of $f$ in $K$, $(f^{(\alpha)}|_{K})_{\a\in\N^{2}}\in\mathcal{E}_{\{w\}}(K)$ is a surjective map.
\item w is a strong weight function, it is to say, 
$$\lim_{\epsilon\to 0^+}\overline{\lim_{t\to\infty}}\frac{\epsilon w(t)}{w(\epsilon t)}=0.$$
\end{enumerate}
\end{theo}

Let $k_{1}=\frac{1}{4\log|q|}$. We consider the weight function defined by $w_{0}(t)=k_{1}\log^2(t)$ for $t\ge 1$ and $w_{0}(t)=0$ for $0\le t\le1$. As the authors write in~\cite{bonetbraunmeisetaylor}, the value of a weight function near the origin is not relevant for the space of functions generated in the sequel.

The following lemma can be easily verified.
\begin{lemma}
$w_{0}$ is a weight function.
\end{lemma}


Under this definition of $w_{0}$ we have 
$$\phi_{w_{0}}^{\star}(y)=\sup\{xy-\phi_{w_{0}}(x):x\ge 0\}=\sup\{xy-\frac{x^2}{4\log|q|}:x\ge 0\}=\log|q| y^2,\quad y\ge0.$$

The spaces appearing in Definition~\ref{def773} concerning this weight function are the following: for any nonempty compact set $K\subseteq \R^2$,
$\mathcal{E	}_{\{w_{0}\}}(K)$ is the set of $w_{0}$-Whitney jets on $K$, which consists of every jet $F=(f^{\a})_{\a\in\N^{2}}$ on $K$ such that there exist $m\in\N$, $M>0$ with
$$|f^{\a}(x)|\le M|q|^{m|\a|^{2}},\quad x\in K,\a\in\N^2$$
and such that for  every $l\in\N$ and $\a\in\N^2$ with $|\a|\le l$ we have
$$|(R_{x}^{l}F)_{\a}(y)|\le M\frac{|x-y|^{l+1-|\a|}}{(l+1-|\a|)!}|q|^{m(l+1)^2},\quad x,y\in K. $$

We derive that $\mathcal{E}_{\{w_{0}\}}(K)$ consists of the Whitney jets on $K$ such that there exist $C_1,H>0$ with
\begin{equation}\label{e820}
|f^{\a}(x)|\le C_{1}H^{|\a|}|q|^{A\frac{|\a|^2}{2}},\quad x\in K,\a\in\N^2,
\end{equation}
and for every $x,y\in K$ and all $l\in\N,\a\in\N^2$ with $|\a|\le l$
\begin{equation}\label{e824}
|(R_{x}^{l}F)_{\a}(y)|\le C_{1}H^l|q|^{A\frac{l^2}{2}}\frac{|x-y|^{l+1-|\a|}}{(l+1-|\a|)!}.
\end{equation}
\begin{theo}\label{teo827}
$w_{0}$ is a strong weight function so that Theorem~\ref{teo792} holds.
\end{theo}
\begin{proof}


$$\lim_{\epsilon\to0^+}\lim_{t\to\infty}\frac{\epsilon w(t)}{w(\epsilon t)}=\lim_{\epsilon\to 0^+}\lim_{t\to\infty}\frac{\epsilon k_{1}\log^2(t)}{k_{1}\log^2(\epsilon t)}=\lim_{\epsilon\to0^+}\epsilon=0.$$
\end{proof}

\textbf{Remark:} A continuous function $f$ which is $w_{0}-\mathcal{C}^{\infty}$ in the sense of Whitney on a compact set $K$ is indeed $\mathcal{C}^{\infty}$ in the usual sense in $\hbox{Int}(K)$ and verifies $q-$Gevrey bounds of the same type. Moreover, we have
$$f^{k}(x,y)=\partial_{x}^{k_{1}}\partial_{y}^{k_{2}}f(x,y),$$
for every $k=(k_{1},k_{2})\in\N^2$ and $(x,y)\in\hbox{Int}(K)$.

Next result is an adaptation of Lemma 4.1.2 in~\cite{ramissauloyzhang}. Here, we need to determine bounds in order to achieve a $q-$Gevrey type result. 
\begin{lemma}\label{lema808}
Let $U$ be an open set in $\C^{\star}$ and $f:Uq^{-\N}\to\mathbb{H}$ a holomorphic function with $\hat{f}=\sum_{h\ge0}a_h\epsilon^h\in\mathbb{H}[[\epsilon]]$ being its $q$-Gevrey asymptotic expansion of type $A>0$ in $Uq^{-\N}$. Then, for any $n\in\N$, the family $\partial_{\epsilon}^{n}f(\epsilon)$ of $n-$complex derivatives of $f$ satisfies that for every compact set $K\subseteq U$ and $k,m\in\N$ with $k\le m$, there exist $C_1,H>0$ such that 
\begin{equation}\label{e810}
\left\|\partial_{\epsilon}^{k}f(\epsilon_{a})-\sum_{h=0}^{m-k}\frac{\partial_{\epsilon}^{k+h}f(\epsilon_{b})}{h!}(\epsilon_{a}-\epsilon_{b})^{h}\right\|_{\mathbb{H}}\le C_{1}H^{m}|q|^{A\frac{m^2}{2}}\frac{|\epsilon_{a}-\epsilon_{b}|^{m+1-k}}{(m+1-k)!},
\end{equation}
for every $\epsilon_{a},\epsilon_{b}\in Kq^{-\N}\cup\{0\}$. Here, we write $\partial^{l}_{\epsilon}f(0)=l!a_{l}$ for $l\in\N$.
\end{lemma}
\begin{proof}
We will first state the result when $\epsilon_{b}=0$. Indeed, we prove in this first step that the family of functions with $q-$Gevrey asymptotic expansion of type $A>0$ in a fixed $q-$spiral is closed under derivation. Proposition~\ref{prop3} turns out to be a particular case of this result.\\ 
Let $m\in\N$, $K$ be a compact set in $U$ and consider another compact set $K_{1}$ such that $K\subseteq K_{1}\subseteq U$. We define 
$$R_{m}(\epsilon):=\epsilon^{-m-1}(f(\epsilon)-\sum_{h=0}^{m}\frac{\partial^{h}_{\epsilon}f(0)}{h!}\epsilon^{h}),\quad \epsilon\in Kq^{-\N},$$
where $\partial_{\epsilon}^{h}f(0)$ denotes the limit of $\partial_{\epsilon}^{h}f(\epsilon)$ for $\epsilon\in Kq^{-\N}$. Then we have that 
\begin{equation}\label{e00}
\partial_{\epsilon}f(\epsilon)=\sum_{h=1}^{m}\frac{\partial_{\epsilon}^{h}f(0)}{h!}h\epsilon^{h-1}+(\partial_{\epsilon}R_{m}(\epsilon))\epsilon^{m+1}+(m+1)R_{m}(\epsilon)\epsilon^{m}.
\end{equation}
Moreover, from Definition~\ref{defi679}, there exist $C,H>0$ such that $\left\|R_{m}(\epsilon)\right\|\le C H^{m}\frac{|q|^{A\frac{m^2}{2}}}{(m+1)!}$ for every $\epsilon\in K_1q^{-\N}$.

\begin{lemma}[Lemma 4.4.1~\cite{ramissauloyzhang}]
There exists $\rho>0$ such that $\overline{D}(\epsilon,\rho|\epsilon|)\subseteq K_{1}q^{-\N}$ for every $\epsilon\in Kq^{-\N}$.
\end{lemma}

Cauchy's integral formula and $q-$Gevrey expansion of $f$ guarantee the existence of a positive constant $C_{2}>0$ such that $$\left\|\partial_{\epsilon}R_{m}(\epsilon)\right\|_{\mathbb{H}}\le C_{2}H^{m}\frac{|q|^{A\frac{m^2}{2}}}{(m+1)!}\frac{1}{\rho|\epsilon|},\quad\epsilon\in Kq^{-\N},$$ 
This yields the existence of $C_{3}>0$ such that
\begin{align*}
\left\|\epsilon^{-m}(\partial_{\epsilon}f(\epsilon)-\sum_{h=0}^{m-1}\frac{\partial_{\epsilon}^{h+1}f(0)}{h!}\epsilon^{h})\right\|_{\mathbb{H}}&\le\left\|\partial_{\epsilon}R_{m}(\epsilon)\right\|_{\mathbb{H}}|\epsilon|+(m+1)\left\|R_{m}(\epsilon)\right\|_{\mathbb{H}}\\
&\le C_{2}A_{1}^{m}\frac{|q|^{A\frac{m^2}{2}}}{m!},\quad\epsilon\in Kq^{-\N}.
\end{align*}
An induction reasoning is sufficient to conclude the proof for every $m\ge 0$.

We now study the case where $\epsilon_{b}\neq 0$ and only give details for $k=0$. For $k\ge 1$ one only has to take into account that the derivatives of $f$ also admit $q-$Gevrey asymptotic expansion of type $A$ and consider the function $\partial_{\epsilon}^{k}f$.\\
If $\epsilon_{b}\neq0$ we treat two cases:\\
If $|\epsilon_{a}-\epsilon_{b}|\le \rho|\epsilon_{b}|$, then $[\epsilon_{a},\epsilon_{b}]$ is contained in $K_{1}q^{-\N}$ and we conclude from Cauchy's integral formula.

If $|\epsilon_{a}-\epsilon_{b}|>\rho|\epsilon_{b}|$, then we bear in mind that the result is obvious when $f$ is a polynomial and write $f(\epsilon)=\epsilon^{m+1}R_{m}(\epsilon)+p(\epsilon)$ where $p(\epsilon)=\sum_{h=0}^{m}\frac{\partial_{\epsilon}^{h}f(0)}{h!}\epsilon^h$. So,  it is sufficient to prove (\ref{e810}) when $f(\epsilon):=\epsilon^{m+1}R_{m}(\epsilon)$.
The result follows from $q-$Gevrey bounds for $\left\|\partial_{\epsilon}^{k}R_{m}\right\|_{\mathbb{H}}$, $k=0,...,n$ and usual estimates.
\end{proof}

The following lemma generalizes Lemma 6 in~\cite{malek}.
\begin{lemma}\label{lema879}
Let $f:Uq^{-\N}\to\mathbb{H}$ be a holomorphic function having $\hat{f}(\epsilon)=\sum_{h\ge0}a_{h}\epsilon^h\in\mathbb{H}[[\epsilon]]$ as its $q-$Gevrey asymptotic expansion of type $A>0$ on $Uq^{-\N}$. Let $K\subseteq U$ be a compact set. Then, the function $(\epsilon_{1},\epsilon_{2})\mapsto \phi(\epsilon_{1}+i\epsilon_{2})=f(\epsilon_{1},\epsilon_{2})$ is a $w_{0}-\mathcal{C}^{\infty}$ function in the sense of Whitney on the compact set 
$$ K'=\{(\epsilon_{1},\epsilon_{2})\in\R^2:\epsilon_{1}+i\epsilon_{2}\in Kq^{-\N}\cup\{0\}\}.$$
\end{lemma}
\begin{proof}
We consider the set of functions $(\phi^{(k_{1},k_{2})})_{(k_{1},k_{2})\in\N^2}$ defined by 
\begin{equation}\label{aux6}
\phi^{(k_{1},k_{2})}:=i^{k_{2}}\partial_{\epsilon}^{k_{1}+k_{2}}f(\epsilon),\quad (k_{1},k_{2})\in\N^2,(\epsilon_{1},\epsilon_{2})\in K'.
\end{equation}
From Lemma~\ref{lema808}, function $f$ satisfies bounds as in (\ref{e810}). Written in terms of the elements in $(\phi^{(k_{1},k_{2})})_{(k_{1},k_{2})\in\N^2}$ we have the existence of $C_{1},H>0$ such that for every $(k_{1},k_{2})\in\N^2$, $m\ge 0$
\begin{align*}
&\left\|\frac{1}{i^{k_{2}}}\phi^{(k_{1},k_{2})}(x_{1},y_{1})-\sum_{p=0}^{m-|(k_{1},k_{2})|}\sum_{h_{1}+h_{2}=p}\frac{\phi^{(k_{1}+h_{1},k_{2}+h_{2})}(x_{2},y_{2})}{i^{k_{2}+h_{2}}p!}\right.\\
&\left. \times\frac{p!}{h_{1}!h_{2}!}(x_{1}-x_{2})^{h_{1}}i^{h_{2}}(y_{1}-y_{2})^{h_{2}}\right\|_{\mathbb{H}}\le C_{1}H^{m}|q|^{A\frac{m^2}{2}}\frac{\left\|(x_{1}-x_{2},y_{1}-y_{2})\right\|_{\R^2}^{m+1-|(k_{1},k_{2})|}}{(m+1-|(k_{1},k_{2})|)!}
\end{align*}
for $(x_{1},y_{1}),(x_{2},y_{2})\in K'$. Expression (\ref{e820}) can be directly checked from (\ref{aux6}) and (\ref{e810}) for $\epsilon_{b}=0$ and $m=k$. This yields the set $(\phi^{(k_{1},k_{2})})_{(k_{1},k_{2})\in\N^2}$ is an element in $\mathcal{E}_{\{w_{0}\}}(K')$
\end{proof}

Next result allows us to glue together a finite number of jets in $\mathcal{E}_{\{w_{0}\}}(K)$, for a given compact set $K$.
\begin{theo}\label{lemakantor}[\cite{kantor}. Theorem II.1.3]
Let $K_{1},K_{2}$ be compact sets in $\R^2$. The following statements are equivalent:
\begin{enumerate}
\item[i.] The sequence
$$0\longrightarrow\mathcal{E}_{\{w_{0}\}}(K_{1})\stackrel{\pi}{\longrightarrow}\mathcal{E}_{\{w_{0}\}}(K_{1})\oplus\mathcal{E}_{\{w_{0}\} }(K_{2})\stackrel{\delta}{\longrightarrow}\mathcal{E}_{\{w_{0}\}}(K_{1}\cap K_{2})\longrightarrow 0$$
is exact. $\pi(f)=(f|_{K_{1}},f|_{K_{2}})$ and $\delta(f,g)=f|_{K_{1}\cap K_{2}}-g|_{K_{1}\cap K_{2}}$.
\item[ii.]  Let $f_{1}\in\mathcal{E}_{\{w_{0}\}}(K_{1})$ and $f_{2}\in\mathcal{E}_{\{w_{0}\}}(K_{2})$ be such that $f_{1}(x)=f_{2}(x)$ for every $x\in K_{1}\cap K_{2}$. The function $f$ defined by $f(x)=f_{1}(x)$ if $x\in K_{1}$ and $f(x)=f_{2}(x)$ if $x\in K_{2}$ belongs to $\mathcal{E}_{\{w_{0}\}}(K_{1}\cup K_{2})$.
\item[iii.] If $K_{1}\cap K_{2}\neq\emptyset$ then there exist $A_{3},A_{4}>0$ such that
$$\overline{M}(A_{3}\hbox{dist}(x,K_{1}\cap K_{2})\le A_{4}\overline{M}(\hbox{dist}(x,K_{2})),$$
for every $x\in K_{1}$. Here, $\overline{M}$ denotes the function given by $\overline{M}(0)=0$ and $\overline{M}(t)=\inf_{n\in\N}t^nM_{n}$ for $t>0$. 
$\hbox{dist}(x,K)$ stands for the distance from $x$ to set $K$.

\end{enumerate}
\end{theo}
\begin{corol}\label{coro910}[\cite{ramissauloyzhang}, Lemma 4.3.6]
Given $\tilde{K}_{1},\tilde{K}_{2}$ nonempty compact sets in $\C^{\star}$, if we put $K_{i}:=\{(\epsilon_{1},\epsilon_{2})\in\R^{2}:\epsilon_{1}+i\epsilon_{2}\in\tilde{K}_{i}q^{-\N}\cup\{0\}\}$, $i=1,2$, then the previous theorem holds for $K_{1}$ and $K_{2}$.
\end{corol}

As the authors remark in~\cite{ramissauloyzhang}, condition $iii)$ in the previous result is known as transversality condition which is more constricting than {\L}ojasiewicz's condition (see~\cite{malgrange1}).

Next proposition is devoted to show that the cocycle constructed in Proposition~\ref{prop754} splits in the space of $w_{0}-\mathcal{C}^{\infty}$ functions in the sense of Whitney. Whitney-type extension results on $\mathcal{E}_{\{w_{0}\}}(K)$ (Theorem~\ref{teo792} and Theorem~\ref{teo827}) will play an important role in the following step.
\begin{prop}\label{prop896}
Let $(U_{I}q^{-\N})_{I\in\mathcal{I}}$ be a good covering and let $(g_{I,I'}(\epsilon))_{(I,I')\in\mathcal{I}^{2}}$ be the $q-$Gevrey $\mathbb{H}_{\mathcal{T},\rho}$-cocycle of type $\tilde{A}$ constructed in Proposition~\ref{prop754}. We choose a family of compact sets $K_{I}\subseteq U_{I}$ for $I\in\mathcal{I}$, with $\hbox{Int}(K_{I})\neq\emptyset$, in such a way that $\cup_{I\in\mathcal{I}}(K_{I}q^{-\N})$ is $\mathcal{U}\setminus\{0\}$, where $\mathcal{U}$ is a neighborhood of 0 in $\C$.

Then, for all $I\in\mathcal{I}$, there exists a $w_{0}-\mathcal{C}^{\infty}$ function $f_{I}(\epsilon_{1},\epsilon_{2})$ in the sense of Whitney on the compact set $A_{I}=\{(\epsilon_1,\epsilon_2)\in\R^{2}:\epsilon_{1}+i\epsilon_{2}\in K_{I}q^{-\N}\cup\{0\}\}$, with values in the Banach space $\mathbb{H}_{\mathcal{T},\rho}$, such that 
\begin{equation}\label{e888}
g_{I,I'}(\epsilon_{1}+i\epsilon_{2})=f_{I'}(\epsilon_{1},\epsilon_{2})-f_{I}(\epsilon_{1},\epsilon_{2})
\end{equation}
for all $I,I'\in\mathcal{I}$ such that $A_{I}\cap A_{I'}\neq\emptyset$ and, for every $(\epsilon_{1},\epsilon_{2})\in A_{I}\cap A_{I'}$.
\end{prop}
\begin{proof}
The proof follows similar arguments as Lemma 3.12 in~\cite{ramissauloyzhang} and it is an adaptation of Proposition 5 in~\cite{malek} under $q-$Gevrey settings.

Let $I,I'\in\mathcal{I}$ such that $A_{I}\cap A_{I'}\neq\emptyset$. From Lemma~\ref{lema879}, we have the function $(\epsilon_{1},\epsilon_{2})\mapsto g_{I,I'}(\epsilon_{1}+i\epsilon_{2})$  is a $w_{0}-\mathcal{C}^{\infty}$ function in the sense of Whitney on $A_{I}\cap A_{I'}$. In the following we provide the construction of $f_{I}$ for $I\in\mathcal{I}$ verifying (\ref{e888}).

Let us fix any $I\in\mathcal{I}$. We consider any $w_{0}-\mathcal{C}^{\infty}$ function in the sense of Whitney on $A_{I}$. By definition of the good covering $(U_{I}q^{-\N})_{I\in\mathcal{I}}$ we are under the following cases:

Case 1: If there is at least one $I'\in\mathcal{I}$, $I\neq I'$, such that $A_{I}\cap A_{I'}\neq\emptyset$ but $A_{I}\cap A_{I'}\cap A_{I''}=\emptyset$ for every $I''\in\mathcal{I}$ with $I''\neq I'\neq I$, then we define $e_{I,I'}(\epsilon_{1},\epsilon_{2})=f_{I}(\epsilon_{1},\epsilon_{2})+g_{I,I'}(\epsilon_{1}+i\epsilon_{2})$ for every $(\epsilon_{1},\epsilon_{2})\in A_{I}\cap A_{I'}$.  $e_{I,I'}$ is a $w_{0}-\mathcal{C}^{\infty}$ function in the sense of Whitney in $A_{I}\cap A_{I'}$. From Theorem~\ref{teo792} and Theorem~\ref{teo827}, we can extend $e_{I,I'}$ to a $w_{0}-\mathcal{C}^{\infty}$ function in the sense of Whitney on $A_{I'}$. This extension is called $f_{I'}$. We have
$$g_{I,I'}(\epsilon_{1}+i\epsilon_{2})=f_{I'}(\epsilon_{1},\epsilon_{2})-f_{I}(\epsilon_{1},\epsilon_{2}), \quad (\epsilon_{1},\epsilon_{2})\in A_{I}\cap A_{I'}.$$

Case 2: There exist two different $I',I''\in\mathcal{I}$ with $I'\neq I\neq I''$ such that $A_{I}\cap A_{I'}\cap A_{I''}\neq\emptyset$. We first construct a $w_{0}-\mathcal{C}^{\infty}$ function in the sense of Whitney on $A_{I'}$, $f_{I'}(\epsilon_{1},\epsilon_{2})$, verifying
\begin{equation}\label{e912}
g_{I,I'}(\epsilon_{1}+i\epsilon_{2})=f_{I'}(\epsilon_{1},\epsilon_{2})-f_{I}(\epsilon_{1},\epsilon_{2}),\quad (\epsilon_{1},\epsilon_{2})\in A_{I}\cap A_{I'}.
\end{equation}
We define $e_{I,I''}(\epsilon_{1},\epsilon_{2})=f_{I}(\epsilon_{1},\epsilon_{2})+g_{I,I''}(\epsilon_{1}+i\epsilon_{2})$ for every $(\epsilon_{1},\epsilon_{2})\in A_{I}\cap A_{I''}$ and $e_{I',I''}(\epsilon_{1},\epsilon_{2})=f_{I'}(\epsilon_{1},\epsilon_{2})+g_{I',I''}(\epsilon_{1}+i\epsilon_{2})$ whenever $(\epsilon_{1},\epsilon_{2})\in A_{I'}\cap A_{I''}$. From (\ref{e912}) we have $e_{I,I''}(\epsilon_{1},\epsilon_{2})=e_{I',I''}(\epsilon_{1},\epsilon_{2})$ for every $(\epsilon_{1},\epsilon_{2})\in A_{I}\cap A_{I'}\cap A_{I''}$. From this, we can define  
$$
e_{I''}(\epsilon_{1},\epsilon_{2}):= \left\{ \begin{array}{cc}
             e_{I,I''}(\epsilon_{1},\epsilon_{2}) &   \hbox{ if }  (\epsilon_{1},\epsilon_{2})\in A_{I}\cap A_{I''} \\
             e_{I',I''}(\epsilon_{1},\epsilon_{2}) &   \hbox{ if }  (\epsilon_{1},\epsilon_{2})\in A_{I'}\cap A_{I''}.
             \end{array}
   \right.
$$
From Theorem~\ref{lemakantor} and Corollary~\ref{coro910} we deduce $e_{I''}(\epsilon_{1},\epsilon_{2})$ can be extended to a $w_{0}-\mathcal{C}^{\infty}$ function in the sense of Whitney in $A_{I''}$, $f_{I''}(\epsilon_{1},\epsilon_{2})$. It is straightforward to check, from the way $f_{I''}$ was constructed, that $f_{I''}(\epsilon_{1},\epsilon_{2})=f_{I}(\epsilon_{1},\epsilon_{2})+g_{I,I''}(\epsilon_{1}+i\epsilon_{2})$ when $(\epsilon_{1},\epsilon_{2})\in A_{I}\cap A_{I''}$ and also $f_{I''}(\epsilon_{1},\epsilon_{2})=f_{I'}(\epsilon_{1},\epsilon_{2})+g_{I',I''}(\epsilon_{1}+i\epsilon_{2})$ for $(\epsilon_{1},\epsilon_{2})\in A_{I'}\cap A_{I''}$.

These two cases solve completely the problem due to nonempty intersection of four different compacts in $(A_{I})_{I\in\mathcal{I}}$ is not allowed when working with a good covering. The functions in $(f_{I})_{I\in\mathcal{I}}$ satisfy (\ref{e888}).
\end{proof}
\end{section}

\begin{section}{Existence of formal series solutions and $q-$Gevrey expansions}

In the current section we set the main result in this work. We establish the existence of a formal power series with coefficients belonging to $\mathbb{H}_{\mathcal{T},\rho}$ which asymptotically represents the actual solutions found in Theorem~\ref{teo572} for the problem (\ref{ee1})+(\ref{ii1}). Moreover, each actual solution turns out to admit this formal power series as $q-$Gevrey expansion of a certain type in the $q-$spiral where the solution is defined.

The following lemma will be useful in the following. We only sketch its proof. For more details we refer to~\cite{narasimhan}.
\begin{lemma}\label{lemafin}
Let $U$ be an open and bounded set in $\R^2$. We consider $h\in\mathcal{C}^{\infty}(U)$ (in the classical sense) verifying bounds as in (\ref{e820}) and (\ref{e824}) for every $(\epsilon_{1},\epsilon_2)\in U$. Let $g$ be the solution of the equation 
\begin{equation}\label{probdelta}
\partial_{\overline{\epsilon}}g(\epsilon_{1},\epsilon_{2}):=\frac{1}{2}(\partial_{\epsilon_{1}}+i\partial_{\epsilon_{2}})g(\epsilon_{1}+i\epsilon_{2})=h(\epsilon_{1},\epsilon_{2}),\quad (\epsilon_{1},\epsilon_{2})\in U.
\end{equation}
Then $g$ also verifies bounds in the nature of (\ref{e820}) and (\ref{e824}) for $(\epsilon_{1},\epsilon_{2})\in U$.
\end{lemma}
\begin{proof}
Let $h_1$ be any extension of the function $h$ to $\R^2$ with compact support which preserves bounds in (\ref{e820}) and (\ref{e824}) in $\R^2$. We have 
$$g(\epsilon_{1},\epsilon_{2}):=-\frac{1}{\pi}\int_{\R^2}\frac{h_{1}(x)}{x-\epsilon}d\xi d\eta,\quad (\epsilon_{1},\epsilon_{2})\in U$$
solves (\ref{probdelta}). Here, $\epsilon=(\epsilon_{1},\epsilon_{2})$, $x=(\xi,\eta)$ and $d\xi d\eta$ stands for Lebesgue measure in $x-$plane.
Bounds in (\ref{e820}) for the function $g$ come out from
$$\frac{\partial^{\a_{1}+\a_{2}}g}{\partial\epsilon_{1}^{\a_1}\partial\epsilon_{2}^{\a_{2}}}(\epsilon_{1},\epsilon_{2})=-\frac{1}{\pi}\int_{\R^2}\frac{\partial^{\a_{1}+\a_{2}}h_{1}}{\partial\epsilon_{1}^{\a_1}\partial\epsilon_{2}^{\a_{2}}}(x)\frac{1}{x-\epsilon}d\xi d\eta,$$
for every $\alpha=(\alpha_{1},\alpha_{2})\in\N^2$ and $(\epsilon_{1},\epsilon_{2})\in U$, and from the fact that the function $x=(x_{1},x_{2})\mapsto 1/|x|$ is Lebesgue integrable in any compact set containing 0.

On the other hand, $g$ satisfies estimates in (\ref{e824}) from Taylor formula with integral remainder.
\end{proof}

We now give a decomposition result of the functions $X_{I}$ constructed in Theorem~\ref{teo572}. The procedure is adapted from~\cite{malek} under $q-$Gevrey settings.
For every $I\in\mathcal{I}$, we denote $X_{I}(\epsilon):U_{I}q^{-\N}\to\mathbb{H}_{\mathcal{T},\rho}$ the holomorphic function given by $X_{I}(\epsilon):=(t,z)\mapsto X_{I}(\epsilon,t,z)$. 
\begin{prop}\label{prop945}
There exists a $w_{0}-\mathcal{C}^{\infty}$ function $u(\epsilon_{1},\epsilon_{2})$ and a holomorphic function $a(\epsilon_{1}+i\epsilon_{2})$ defined on the neighborhood of 0 $\hbox{Int}(\cup_{I\in\mathcal{I}}A_{I})$ such that 
\begin{equation}\label{e947}
X_{I}(\epsilon_{1}+i\epsilon_{2})=f_{I}(\epsilon_{1},\epsilon_{2})+u(\epsilon_{1},\epsilon_{2})+a(\epsilon_{1}+i\epsilon_{2}),\quad (\epsilon_{1},\epsilon_{2})\in\hbox{Int}(A_{I}),
\end{equation} 
for every $I\in\mathcal{I}$.
\end{prop}
\begin{proof}
From the definition of the cocycle $(g_{I,I'})_{(I,I')\in\mathcal{I}^{2}}$ in Proposition~\ref{prop754} and from Proposition~\ref{prop896} we derive
$$X_{I}(\epsilon_{1}+i\epsilon_{2})-f_{I}(\epsilon_{1},\epsilon_{2})=X_{I'}(\epsilon_{1}+i\epsilon_{2})-f_{I'}(\epsilon_{1},\epsilon_{2}),\quad (\epsilon_{1},\epsilon_{2})\in A_{I}\cap A_{I'}\setminus\{(0,0)\},$$
whenever $(I,I')\in\mathcal{I}^{2}$ and $A_{I}\cap A_{I'}\neq\emptyset$. The function $X-f$ given by 
$$(X-f)(\epsilon_{1},\epsilon_{2}):=X_{I}(\epsilon_{1}+i\epsilon_{2})-f_{I}(\epsilon_{1},\epsilon_{2}),\quad (\epsilon_{1},\epsilon_{2})\in A_{I}\setminus\{(0,0)\}$$
is well defined on $W\setminus\{(0,0)\}$, where $W=\cup_{I\in\mathcal{I}}A_{I}$ is a closed neighborhood of $(0,0)$.

For every $I\in\mathcal{I}$, $X_{I}$ is a holomorphic function on $U_{I}q^{-\N}$ so that Cauchy-Riemann equations hold:
$$\partial_{\overline{\epsilon}}(X_{I})(\epsilon_{1}+i\epsilon_{2})=0,\quad (\epsilon_{1},\epsilon_{2})\in A_{I}\setminus\{(0,0)\}.$$
This yields $\partial_{\overline{\epsilon}}(X-f)(\epsilon_{1},\epsilon_{2})=-\partial_{\overline{\epsilon}}f_{I}(\epsilon_{1},\epsilon_{2})$ for every $I\in\mathcal{I}$ and $(\epsilon_{1},\epsilon_{2})\in\hbox{Int}(A_{I})$. 

We have $-\partial_{\overline{\epsilon}}f_{I}(\epsilon_{1},\epsilon_{2})$ can be extended to a $w_{0}-\mathcal{C}^{\infty}$ function in the sense of Whitney on $A_{I}$. This yields $f_{I}$ is $w_{0}-\mathcal{C}^{\infty}$ in the sense of Whitney on $A_{I}$. In fact, their $q-$Gevrey type coincide.

From this, we deduce that $\partial_{\overline{\epsilon}}(X-f)$ is a $w_{0}-\mathcal{C}^{\infty}$ function in the sense of Whitney on $A_{I}$ for every $I\in\mathcal{I}$  and also that $\partial_{\overline{\epsilon}}f_{I}(\epsilon_{1},\epsilon_{2})=\partial_{\overline{\epsilon}}f_{I'}(\epsilon_{1},\epsilon_{2})$ for every $(\epsilon_{1},\epsilon_{2})\in\hbox{Int}(A_{I}\cap A_{I'})$ and every $I,I'\in\mathcal{I}$ due to $g_{I,I'}(\epsilon)$ is a holomorphic function on $U_{I}q^{-\N}\cap U_{I'}q^{-\N}$. The previous equality is also true for $(\epsilon_{1},\epsilon_{2})\in A_{I}\cap A_{I'}$ from the fact that $f_{I}$ is $w_{0}-\mathcal{C}^{\infty}$ in the sense of Whitney on $A_{I}$.

From Theorem~\ref{lemakantor} and Corollary~\ref{coro910} we derive $\partial_{\overline{\epsilon}}(X-f)$ is a $w_{0}-\mathcal{C}^{\infty}$ function in the sense of Whitney on $\cup_{I\in\mathcal{I}}A_{I}$.

Taking into account Lemma~\ref{lemafin} we derive the existence of a $\mathcal{C}^{\infty}$ function $u(\epsilon_{1},\epsilon_{2})$ in the usual sense, defined in $\hbox{Int}(W)$ verifying $q-$Gevrey bounds of a certain positive type, such that 
$$\partial_{\overline{\epsilon}}u(\epsilon_{1},\epsilon_{2})=\partial_{\overline{\epsilon}}(X-f)(\epsilon_{1},\epsilon_{2}),\quad (\epsilon_{1},\epsilon_{2})\in\hbox{Int}(W).$$
From this last expression we have $u(\epsilon_{1},\epsilon_{2})-(X-f)(\epsilon_{1},\epsilon_{2})$ defines a holomorphic function on $\hbox{Int}(W)\setminus\{(0,0)\}$.

For every $I\in\mathcal{I}$, $X_{I}$ is a bounded $\mathbb{H}_{\mathcal{T},\rho}-$function in $\hbox{Int}(W)\setminus\{(0,0)\}$, and so it is the function $u(\epsilon_{1},\epsilon_{2})-(X-f)(\epsilon_{1},\epsilon_{2})$. The origin turns out to be a removable singularity so the function $u(\epsilon_{1},\epsilon_{2})-(X-f)(\epsilon_{1},\epsilon_{2})$ can be extended to a holomorphic function defined on $\hbox{Int}(W)$. The result follows from here.
\end{proof}
We are under conditions to enunciate the main result in the present work.

\begin{theo}\label{teopral}
Under the same hypotheses as in Theorem~\ref{teo572}, there exists a formal power series 
$$\hat{X}(\epsilon,t,z)=\sum_{k\ge0}\frac{X_{k}(t,z)}{k!}\epsilon^k\in\mathbb{H}_{\mathcal{T},\rho}[[\epsilon]],$$
formal solution of
\begin{equation}\label{e949}
\epsilon t\partial_{z}^{S}\hat{X}(\epsilon,qt,z)+\partial_{z}^{S}\hat{X}(\epsilon,t,z)=\sum_{k=0}^{S-1}b_{k}(\epsilon,z)(t\sigma_{q})^{m_{0,k}}(\partial_{z}^{k}\hat{X})(\epsilon,t,zq^{-m_{1,k}}).
\end{equation}
 Moreover, let $I\in\mathcal{I}$ and $\tilde{K}_{I}$ any compact subset of $\hbox{Int}(K_{I})$. There exists $B>0$ such that the function $X_{I}(\epsilon,t,z)$ constructed in Theorem~\ref{teo572} admits $\hat{X}(\epsilon,t,z)$ as its $q-$Gevrey asymptotic expansion of type $B$ in $\tilde{K}_{I}q^{-\N}$.
\end{theo}
\begin{proof}
Let $I\in\mathcal{I}$ and $\tilde{K}_{I}$ any compact subset of $\hbox{Int}(K_{I})$.

From Proposition~\ref{prop945} we can extend $X_{I}(\epsilon_{1}+i\epsilon_{2})$ to a $w_{0}-\mathcal{C}^{\infty}$ function in the sense of Whitney on $\tilde{A}_{I}=\{(\epsilon_{1},\epsilon_{2})\in\R^2:\epsilon_{1}+i\epsilon_{2}\in\tilde{K}_{I}q^{-\N}\cup\{0\} \}\subseteq \hbox{Int}(A_{I})\cup\{(0,0)\}$. Let us fix $I\in\mathcal{I}$. We consider the family $(X^{(h_{1},h_{2})}(\epsilon_{1},\epsilon_{2}))_{(h_{1},h_{2})\in\N^{2}}$ associated to $X_{I}$ by Definition~\ref{def773}. We have
$$X_{I}^{(h_{1},h_{2})}(\epsilon_{1},\epsilon_{2})=\partial_{\epsilon_{1}}^{h_{1}}\partial_{\epsilon_{2}}^{h_{2}}X_{I}(\epsilon_{1}+i\epsilon_{2})=i^{h_{2}}\partial_{\epsilon}^{h_{1}+h_{2}}X_{I}(\epsilon),\quad (\epsilon_{1},\epsilon_{2})\in\tilde{A}_{I}\setminus\{(0,0)\},$$
due to $X_{I}(\epsilon)$ is holomorphic on $\hbox{Int}(K_{I})q^{-\N}$.

We have $X_{I}^{(h_{1},h_{2})}(\epsilon_{1},\epsilon_{2})$ is continuous at $(0,0)$ for every $(h_{1},h_{2})\in\N^{2}$ so we can define  for every $k\ge0$
\begin{equation}\label{e1001}
X_{k,I}:=\frac{X^{(h_{1},h_{2})}_{I}(0,0)}{i^{h_{2}}}\in\mathbb{H}_{\mathcal{T},\rho},
\end{equation}
whenever $h_{1}+h_{2}=k$.
Estimates held by any $w_0-\mathcal{C}^{\infty}$ function in the sense of Whitney, (see Definition~\ref{def773} for $\a=(0,0)$) lead us to the existence of positive constants $C_{1},H,B>0$ such that
$$\left\|X_{I}(\epsilon_{1}+i\epsilon_{2})-\sum_{p=0}^{m}\frac{X_{p,I}}{p!}(\epsilon_{1}+i\epsilon_{2})^p\right\|_{\mathbb{H}_{\mathcal{T},\rho}}\le C_{1}H^{m}|q|^{B\frac{m^2}{2}}\frac{|\epsilon_{1}+i\epsilon_{2}|^{m+1}}{(m+1)!}, $$
for every $m\ge0$ and $\epsilon_{1}+i\epsilon_{2}\in\tilde{K}_{I}q^{-\N}$. As a matter of fact, this shows that $X_{I}$ admits $\hat{X}_{I}(\epsilon)=\sum_{k\ge0}\frac{X_{k,I}}{k!}\epsilon^k$ as its $q-$Gevrey expansion of type $B>0$ in $\tilde{K}_{I}q^{-\N}$.

The formal power series $\hat{X}_{I}$ does not depend on $I\in\mathcal{I}$. Indeed, from Theorem~\ref{teo572} we have that $X_{I}(\epsilon)-X_{I'}(\epsilon)$ admits both $\hat{0}$ and $\hat{X}_{I'}-\hat{X}_{I}$ as $q-$asymptotic expansion on $\tilde{K}_{I}q^{-\N}\cap\tilde{K}_{I'}q^{-\N}$ whenever this intersection is not empty. We put $\hat{X}:=\hat{X}_{I}$ for any $I\in\mathcal{I}$. The function $X_{k,I}=X_{k,I}(t,z)\in\mathbb{H}_{\mathcal{T},\rho}$ does not depend on $I$ for every $k\ge0$. We write $X_{k}:=X_{k,I}$ for $k\ge 0$. $X_{I}$ admits $\hat{X}$ as its $q-$Gevrey asymptotic expansion of type $B>0$ in $\tilde{K}_{I}q^{-\N}$ for all $I\in\mathcal{I}$.

In order to achieve the result, it only remains to prove that $\hat{X}(\epsilon,t,z)$ is a formal solution of (\ref{e949}). Let $l\ge 1$. If we derive $l$ times with respect to $\epsilon$ in equation (\ref{e949}) we get that $\partial_{\epsilon}^{l}X_{I}(\epsilon,t,z)$ is a solution of
\begin{equation}\label{e1014}
\epsilon t\partial_{z}^{S}\partial_{\epsilon}^{l}X_{I}(\epsilon,qt,z)+t\partial_{z}^{S}l\partial_{\epsilon}^{l-1}X_{I}(\epsilon,t,z)+\partial_{z}^{S}\partial_{\epsilon}^{l}X_{I}(\epsilon,t,z)\qquad\qquad\qquad\qquad
\end{equation}
$$\qquad\qquad\qquad=\sum_{k=0}^{S-1}\sum_{l_{1}+l_{2}=l}\frac{l!}{l_{1}!l_{2}!}\partial_{\epsilon}^{l_{1}}b_{k}(\epsilon,z)\partial_{\epsilon}^{l_{2}}((t\sigma_{q})^{m_{0,k}})\partial_{z}^{k}X_{I})(\epsilon,t,zq^{-m_{1,k}}).$$
for every $l\ge 1, (t,z)\in\mathcal{T}\times D(0,\rho)$ and $\epsilon\in\tilde{K}_{I}q^{-\N}$. Letting $\epsilon$ to 0 in (\ref{e1014}) we obtain
\begin{equation}\label{e1023}
t\partial_{z}^{S}\frac{X_{l-1}(qt,z)}{(l-1)!}+\partial_{z}^{S}\frac{X_{l}(t,z)}{l!}=\sum_{k=0}^{S-1}\sum_{l_{1}+l_{2}=l}\frac{\partial_{\epsilon}^{l_{1}}b_{k}(\epsilon,z)|_{\epsilon=0}}{l_{1}!}\frac{((t\sigma_{q})^{m_{0,k}}\partial_{z}^{k}X_{l_{2}})(t,zq^{-m_{1,k}})}{l_{2}!}
\end{equation}
for every $l\ge 1, (t,z)\in\mathcal{T}\times D(0,\rho)$. Holomorphy of $b_{k}(\epsilon,z)$ with respect to $\epsilon$ at 0 derives
\begin{equation}\label{e1027}
b_{k}(\epsilon,z)=\sum_{l\ge 0}\frac{\partial_{\epsilon}^{l}b_{k}(\epsilon,z)|_{\epsilon=0}}{l!}\epsilon^{l},
\end{equation}
for $\epsilon$ near 0 and for every $z\in\C$. Statements (\ref{e1014}) and (\ref{e1023}) conclude $\hat{X}(\epsilon,t,z)=\sum_{k\ge0}X_{k}(t,z)\frac{\epsilon^{k}}{k!}$ is a formal solution of (\ref{e949}). 
\end{proof}
\end{section}

\end{document}